\documentclass[11pt]{amsart}
\usepackage{graphicx}
\usepackage[mathcal]{euscript}
\usepackage{float}
\usepackage{comment}

\restylefloat{figure}

\newtheorem{Theorem}{Theorem}

\newtheorem{Example}[Theorem]{Example}

\newtheorem{theorem}{Theorem}[section]
\newtheorem{lemma}[theorem]{Lemma}
\newtheorem{corollary}[theorem]{Corollary}

\theoremstyle{definition}
\newtheorem{definition}[theorem]{Definition}

\theoremstyle{remark}

\numberwithin{equation}{section}

\newcommand{\ot}{\otimes}
\newcommand{\ra}{\rightarrow}

\newcommand{\BR}{\mathbb{R}}

\newcommand{\BZ}{\mathbb{Z}}

\begin{document}

\title[Knotted Tori in $\BR^4$]{Embedded and Lagrangian Knotted  Tori in $\BR^4$ and Hypercube Homology}

\author[S. Baldridge]{Scott Baldridge}

\thanks{S. Baldridge was partially supported by NSF Career Grant DMS-0748636.}

\address{Department of Mathematics, Louisiana State University \newline
\hspace*{.375in} Baton Rouge, LA 70817, USA} \email{\rm{sbaldrid@math.lsu.edu}}

\subjclass{}
\date{August 15, 2010}

\begin{abstract}
In this paper we introduce a representation of a embedded knotted (sometimes Lagrangian) tori in $\BR^4$ called a hypercube diagram, i.e., a 4-dimensional cube diagram.  We prove the existence of hypercube homology that is invariant under 4-dimensional cube diagram moves, a homology that is based on knot Floer homology.  We provide examples of hypercube diagrams and hypercube homology, including using the new invariant to distinguish (up to cube moves) two ``Hopf linked'' tori.  We also give examples of a ``Trefoil'' torus and an immersed knotted torus that is an amalgamation of the $5_2$ knot and a trefoil knot.

\end{abstract}

\maketitle

\bigskip
\section{Introduction}
\bigskip

The study of Lagrangian tori in symplectic 4-manifolds can be traced back to Arnold \cite{Arnold}.   Luttinger \cite{Lut} and Eliashberg-Polterovich \cite{EP} provided significant contributions to this subject in terms of strong restrictions  on  what embedded Lagrangian tori can  exist in $\BR^4$ (see also \cite{EP1}).  Luttinger showed that known constructions of knotted tori in $\BR^4$ do not produce examples of knotted Lagrangian tori.  He questioned whether the unknotted standard torus was the only Lagrangian torus in $\BR^4$.    This question is central to the research started in this paper.  It should be noted that  knotted Lagrangian tori do exist in closed 4-manifolds with $b_2>0$ (cf. \cite{V, FS}), so there is some hope that the same is true in $\BR^4$.
  
\medskip

Another centerpiece of this paper is the development of invariants for smoothly embedded knotted tori in $\BR^4$ in general.  Of course, the study of smoothly embedded knotted surfaces in $\BR^4$ has a rich history (cf. \cite{Fox} and \cite{CS} for example).

\medskip

One of the barriers to studying Lagrangian or embedded knotted tori has been that there are few ways to represent knotted tori in $\BR^4$ that lead to powerful yet easy-to-compute invariants.  In this paper we announce and describe a new  representation of a 2-dimensional knotted tori  in $\BR^4$ called a hypercube diagram.  We also define, state and prove the existence of a homology theory for hypercubes that is clearly invariant under three 4-dimensional hypercube diagram moves.  The hypercube moves are similar in nature to the 3-dimensional cube diagram moves described in \cite{BL}.  The homology theory is based upon knot Floer homology  \cite{PS}.

\medskip

The results of this paper can be nicely summarized by the following two theorems and one example:

\begin{Theorem}
Let $H\Gamma$ be a hypercube diagram.  Then $H\Gamma$ represents a $PL$-torus  that can be smoothed to an immersed Lagrangian torus in $\BR^4$.  If none of the double point circles intersect each other in the $PL$-torus, then the $PL$-torus can be smoothed to an embedded torus in $\BR^4$.  If there are no double point circles, then the $PL$-torus can be smoothed to an embedded Lagrangian torus in $\BR^4$.
\end{Theorem}

Double point circles arise in the construction of the $PL$-torus and are simply intersections of the torus with itself along circles. The only time these  intersections cannot be perturbed away (smoothly) are when two of these circles intersect each other.  

\medskip
 
It was pointed out to the author by Ben McCarty that immersed Lagrangian tori may be profitably thought of as  Lagrangian projections of knotted Legendrian tori in $\BR^5$ (see\cite{Etnyre} for a discussion of Lagrangian projections of Legendrian knots).

\medskip

Like grid diagrams and cube diagrams, a hypercube diagram is a collection of markings in a $4$-dimensional cartesian grid (using $wxyz$-coordinates) that defines a loop or set of loops in $\BR^4$.  We require that these loops, when projected to the $wx$-, $yz$-, $zw$- and $xy$-planes, are grid diagram projections.  The first part of this paper is devoted to carefully defining hypercube diagrams so that the projected grid diagrams make sense and capture important information about the hypercube structure.  This care  is needed.  For example, given any loop in $\BR^4$, the knot projection of the loops under the composition of maps $\pi_z\circ \pi_y:\BR^4 \ra \BR^2$ (each map projects out the coordinate indicated) does not in general equal the knot projection of the loops under the composition of maps $\pi_y \circ \pi_z:\BR^4\ra\BR^2$. In particular, for a hypercube diagram $H\Gamma$, the images of these loops using the maps $\pi_y:\BR^4\ra \BR^3$ and $\pi_w:\BR^4\ra \BR^3$ are usually two different links in $\BR^3$, which we will name  $L_1$ and $L_2$ respectively.  Furthermore, denote the oriented grid projection of $L_1$ onto the $wx$-plane by $G_{wx}$ and the oriented grid projection of $L_2$ onto the $yz$-plane by $G_{yz}$. 

\medskip

The moves on a hypercube diagram are similar to cube diagrams (in fact, the hypercubes and moves should  easily generalize to any dimension, not just 3 or 4).  The first two hypercube moves generate well-defined grid diagram moves on the four projections.  Thus, the first two moves preserve the links $L_1$ and $L_2$ named above.  The difference between cube diagram moves and hypercubes moves is that the third hypercube move swaps the links $L_1$ and $L_2$ (or factors of $L_1$ and $L_2$), i.e., the oriented grid diagram projection $G_{wx}$ becomes $G_{yz}$ and vice versa. 

\medskip

We can now state the second theorem:

\medskip
\begin{Theorem} \label{main_theorem_invariant}
Let $H\Gamma$ be a $4$-dimensional hypercube diagram. Then $CH^-(H\Gamma)$ is a hypercube invariant, that is, it is invariant under all three hypercube moves.  In particular, $$CH^-(H\Gamma) \cong HFK^-(L_1)\ot HFK^-(L_2),$$
where $L_1$ is the link represented by the oriented grid diagram $G_{wx}$ and $L_2$ is the link represented by the oriented grid diagram $G_{yz}$.
\end{Theorem}

In fact, the filtered chain homotopy type of $(C^-(H\Gamma, \partial^-)$ is an invariant under the hypercube moves.  Therefore it seems likely that that there are other invariants that can be derived from the chain complex itself (see for example \cite{PS2, R} or \cite{PST}).

\medskip

Finally, we provide an example that indicates that hypercube diagrams represent a significant subset of all embedded knotted tori in $\BR^4$:

\begin{Example}
There exists a hypercube diagram that represents two smoothly embedded tori in $\BR^4$ such that the two projected links $L_1$ and $L_2$ of the hypercube diagram are both Hopf links (see Example 3 in Section~\ref{exampleofhypercubes}).
\end{Example}

The key words in the example above are smoothly embedded.  It is easy to find a hypercube diagram that represents an immersed torus with two nontrivial knots for $L_1$ and $L_2$ (see Example 4 in Section~\ref{exampleofhypercubes}).   It is more difficult to find such an example that represents an embedded torus (no intersecting double point circles).  This example is important because it shows that there are hypercubes with nontrivial links $L_1$ and $L_2$ that represent smoothly embedded tori.   Furthermore, as explained later, the method used to search for this example indicates how to find a hypercube diagram  that represents an embedded knotted torus such that $L_1$ and $L_2$ are any two specified knots.

\medskip

We also present an example (of the much more common) hypercube diagram that represents two embedded tori where $L_1$ is the Hopf Link and $L_2$ is the split link of two unknots.  We use hypercube homology to distinguish this example from the example above.

\medskip

The organization of this paper is as follows:  In Section 2 we describe oriented grid diagrams and cube diagrams.  Section 3 carefully defines hypercube diagrams and the moves on hypercubes are described in Section 4.  The construction of the $PL$-torus from a hypercube diagram is described in Section 5 together with the proof of Theorem 1.  Hypercube homology is defined in Section 6 including the Alexander polynomial of the hypercube homology.  In the final section we described the examples of embedded and immersed knotted tori discussed above.

\bigskip\medskip

\noindent {\bf Acknowledgements.} \ \ This paper is the realization of an idea that occurred to the author while listening to Peter Ozsv\'{a}th's talk on combinatorial knot Floer theory at Ron Fintushel's 60th birthday conference, so thank you Peter for the inspiring lecture.  I  also thank Paul Kirk and Ben McCarty for helpful conversations.

\section{Background: Oriented Grid Diagrams and Cube Diagrams}

Three dimensional cube diagrams (cf. \cite{BL})  are the first examples of hypercube diagrams in the literature. To understand hypercube diagrams we first need some basic notions about grid diagrams.

\subsection{Grid Diagrams}  Grid diagrams were introduced by Brunn~\cite{Brunn} over a 100 years ago, and discussed more recently in Cromwell~\cite{Cromwell}.  We will extend this idea to {\em oriented grid diagrams}.

\medskip

First, since over- and undercrossings are recorded in knot and link projections, some knowledge of $\BR^3$ is needed to define an oriented grid diagram.  In this case, that information is the choice of an orientation for $\BR^3$.  Let $\BR^3$ be parameterized using $xyz$-coordinates and orient $\BR^3$ using the right-hand rule, given by the ordered basis $x\wedge y \wedge z$.  In the definition that follows, the right-hand orientation is always chosen, however the definition is stated in such a way that it should be easy to see how to define an oriented grid diagram for the lefthand orientation of $\BR^3$  using any choice of labels for the axes.

\medskip

An  {\it oriented grid diagram}  $G_{xy}$ is defined as an $n\times n$ subset of the Cartesian grid  in the $xy$-plane such that each cell contains either a labeled half-integer point called an $X$ marking,  a labeled half-integer point called a $Y$ marking, or an unlabeled point, arranged so that:

\begin{itemize}
\item each row contains exactly one $X$ marking and exactly one $Y$ marking, and\\

\item each column contains exactly one $X$ marking and exactly one $Y$ marking.
\end{itemize}

\medskip

Next, include oriented segments from $X$ to $Y$ markings that are parallel to the $x$- and $y$-axes in the following manner:
\begin{itemize}
\item Connect each $X$ marking to a $Y$ marking with a segment if the segment is parallel to the $x$-axis.  Orient that  segment to go from $X$ to $Y$.  \\
\item Connect each $Y$ marking to an $X$ marking with a segment if the segment is parallel to the $y$-axis.  Orient that segment to go from $Y$ to $X$.
\end{itemize}
An easy way to remember this convention is that each ``vector'' starting at an $X$ marking is parallel to the $x$-axis and each ``vector'' starting at a $Y$ marking is parallel to the $y$-axis (this will be same convention used in later generalizations). The labels $X$ and $Y$ specify which axes the ``vectors'' are parallel to.

\medskip

So far, this configuration gives a link projection with singularities.  In order to resolve the intersection points and completely specify an oriented link, an orientation for the grid diagram itself also needs to be chosen.  Define an {\em orientation} of the grid diagram by choosing an ordering of the $x$- and $y$-axes.  In this paper the notation $G_{xy}$ specifies the order in the usual manner, ie., $G_{xy}$ specifies the ordered basis $x\wedge y$ ($x$ first, $y$ second), $G_{yx}$ specifies the ordered basis $y\wedge x$.  Note the orientation of $\BR^3$ remains unchanged.  With the ordered basis  for the grid chosen, resolve all intersection points by  declaring all segments parallel to the second axis in the order to be ``overcrossings'', where ``overcrossings'' means the segment is the projection of a link in which the $z$-coordinate at the crossing is greater than the $z$-coordinate of the part of the link that projects to the segment parallel to the first axis.  For example, in the grid diagram $G_{xy}$, all oriented segments starting at $Y$ and parallel to the $y$-axis are overcrossings.

\medskip

\begin{figure}[H]
\includegraphics[scale=1]{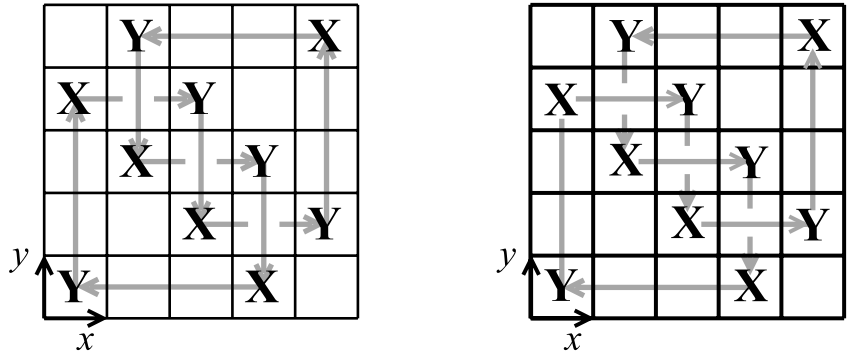}
\caption{\it \small The figure on the left is an $x\wedge y$ oriented grid diagram $G_{xy}$ and the figure on the right is the grid diagram $G_{yx}$, a diagram with the same markings but with the $y\wedge x$ orientation (both use the $x\wedge y\wedge z$-orientation of $\BR^3$).} \label{or}
\end{figure}

Changing the orientation of a  grid diagram changes the link to its mirror.  Also, note that the definition of an oriented grid diagram used in this paper is an improvement over the definition used in \cite{BL} (a careful definition was not needed for that paper).  In that paper, changing the orientation from $G_{xy}$ to $G_{yx}$ also reversed the orientation of the link.  This change matches what is needed for hypercubes.  Finally, set $G_{-yx}=G_{xy}$ following the usual conventions of orientation forms.

\medskip

The definitions for grid diagrams used in the literature (cf. \cite{Audoux}, \cite{DW}, \cite{Karev}, \cite{Levine}, \cite{LOT}, \cite{mos}, \cite{NT}, \cite{PST}, \cite{Vertesi}) usually do not specify axes (and instead refer to horizontal and vertical) and use the letters $O$ and $X$ instead of $X$ and $Y$.  Furthermore, the orientation of the grid is inherited from the orientation of the plane.  But in attempting to generalize this notion to higher dimensions, as the reader will soon observe, the orientation of the grid becomes important.  To connect the definition in this paper to the common definition used in the literature, note that the $x$-axis in $G_{xy}$ is ``horizontal'', the $y$-axis is ``vertical'', and $X$ markings in $G_{xy}$ correspond to $O$ markings, and $Y$ markings in $G_{xy}$ correspond to $X$ markings.

\bigskip

\subsection{Cube Diagrams}  Cube diagrams were first introduced in  Baldridge-Lowrance \cite{BL}.  They are important examples of hypercube diagrams.  Intuitively, a cube diagram can be thought of as an embedding of a knot or link  in a $[0,n]\times [0,n] \times [0,n]$ cube (using $xyz$ coordinates) for some positive integer $n$ such that the link projections of the cube to each axis plane ($x=0$, $y=0$, and $z=0$) are grid diagrams.  The main theorem from that paper generalizes the idea of Reidemeister moves to 3-dimensions:

\begin{theorem}\label{maintheorem1}
Two cube diagrams correspond to ambient isotopic oriented links if and only if one can be obtained from the other by a finite sequence of five cube diagram moves.
\end{theorem}

A cube diagram move is a special ambient isotopy of a link that takes one cube diagram to another cube diagram.   Cube diagram moves differ from other isotopy moves (like triangle moves)  because there are only finitely many instances of only five cube moves that can be performed on any given cube diagram. Hence, one has the desired control similar to that of Reidemeister or Markov moves, but in $3$ dimensions instead of $2$.  Because cube diagrams are inherently 3-dimensional, they can be used to connect differential topology to algebraic topology using the combinatorics of the cube diagram moves.

\medskip

While cube diagrams project to grid diagrams, grid diagrams rarely lift to cube diagrams (which may be the reason they were not discovered earlier). For example, in  Baldridge-McCarty \cite{note} we show that only about 20\% of size 8 grid diagrams of nontrivial knots  lift to cube diagrams, and that this percentage decreases as the size of the grid increases.  This sparsity of cube diagrams leads to strictly stronger invariants than those defined using grid diagrams.  For example, the {\em cube number} of a link is the minimum size cube diagram that represents that link. Similarly, the {\em grid number} (or {\em arc index}) of a link is the minimum size grid diagram that represents that link.  Clearly, the grid number of a link is equal to the grid number of its mirror image: $\mbox{arc index}(G_{xy})=\mbox{arc index}(G_{yx})$ from above. McCarty \cite{private} has shown that cube number can distinguish a knot from its mirror image while grid number cannot. For example, the cube number of the left-handed trefoil is five while the cube number of the right-handed trefoil is seven. Additionally, McCarty in \cite{private} and  \cite{MC2} showed that a modified version of cube number can distinguish between Legendrian knots with the same underlying knot type.

\bigskip

The intuitive definition of cube diagrams above doesn't keep track of orientations of the projected grid diagrams nor how to label the markings.  Both are important in defining the projections to the coordinate axes.  Let $n$ be a positive integer and  let the cube $[0,n]\times [0,n]\times [0,n] \subset \mathbb{R}^3$ be thought of as $3$-dimensional Cartesian grid, i.e., a grid with integer valued vertices with axes $x$, $y$, and $z$.  The number $n$ is called the {\em cube number} or {\em size}.  A \textit{flat} is any right rectangular prism with integer vertices in the cube such that there are two orthogonal edges of length $n$ with the remaining orthogonal edge of length $1$.  A flat with an edge of length 1 that is parallel to the $x$-axis, $y$-axis, or $z$-axis is called an {\em $x$-flat}, {\em $y$-flat}, or {\em $z$-flat} respectively.

\medskip

The embedding of a link using markings in the cube is similar to that grid diagrams.  A marking is a labeled point in $\BR^3$ with half-integer coordinates.  Mark unit cubes in the $3$-dimensional Cartesian grid with either an $X$, $Y$, or $Z$ such that the following {\em marking conditions} hold:

\begin{itemize}
    \item each flat has exactly one $X$, one $Y$, and one $Z$ marking;\\

    \item the markings in each flat forms a right angle such that each ray is parallel to a coordinate axis;\\

    \item for each $x$-flat, $y$-flat, or $z$-flat, the marking that is the vertex of the right angle is an $Z, X,$ or $Y$ marking respectively.
\end{itemize}

\medskip

Like grid diagrams, an oriented link can be embedded into the cube by connecting pairs of markings with a line segment whenever two of their corresponding coordinates are the same.   Each line segment is oriented to go from an $X$ to a $Y$, from a $Y$ to a $Z$, or from a $Z$ to an $X$.   The marking conditions are set up so  the ``vector'' starting at an $X$ marking (line segment going from $X$ to $Y$) is always parallel to the $x$-axis, i.e., an ``$X$ marking vector'' points in the $x$-direction. Similarly, the ``vectors'' at $Y$ and $Z$ markings point in the $y$-direction and $z$-direction respectively.  This convention is analogous to the ``vectors'' described in defining grid diagrams.

\medskip

In $xyz$-coordinates, denote the projection to the $xy$-plane by $\pi_{xy}: \BR^3\ra \BR^2$ given by $\pi_{xy}(x,y,z)=(x,y)$.  Similarly, define $\pi_{yz}(x,y,z)=(y,z)$ and $\pi_{zx}(x,y,z)=(z,x)$.  Arrange the markings in the cube as above so that the following {\em crossing conditions} hold:
\begin{itemize}
\item At every ordinary double point intersection of the link in the $\pi_{xy}$-projection, the segment parallel to the $x$-axis has smaller $z$-coordinate than the segment parallel to the $y$-axis in $\BR^3$.\\

\item At every ordinary double point intersection of the link in the $\pi_{yz}$-projection, the segment parallel to the $y$-axis has smaller $x$-coordinate than the segment parallel to the $z$-axis in $\BR^3$.\\

\item At every ordinary double point intersection point of the link in the $\pi_{zx}$-projection, the segment parallel to the $z$-axis has smaller $y$-coordinate than the segment parallel to the $x$-axis in $\BR^3$.
\end{itemize}

\medskip

If the cube and data structure satisfies both marking and crossing conditions, then it is called a {\it cube diagram}.  We will denote the cube diagram by $\Gamma$ or $\Gamma(L)$ where $L$ is the link that it represents.  Note that the cube itself is canonically oriented by the standard orientation of $\BR^3$---the right hand orientation $x\wedge y\wedge z$.  There is a corresponding definition for left-handed cube diagrams.

\bigskip

\begin{figure}[H]
\includegraphics[scale=.3]{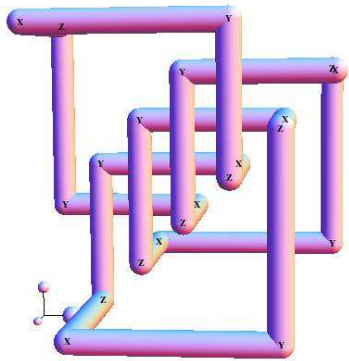}
\caption{Example of a cube diagram of the $8_{19}$ knot. } \label{trefoil}
\end{figure}

\bigskip

Clearly,  the data structure of a cube diagram can be naturally projected onto an oriented grid diagram in a way that respects orientation, overcrossings, and markings.  For example, the projection $\pi_{xy}:\Gamma \ra G_{xy}$ takes $Y$ markings in the cube to $Y$ markings in $G_{xy}$ and takes the $X$ and $Z$ markings in the $\Gamma$ to $X$ markings in $G_{xy}$.  The orientation of the link in $\Gamma$ are the same with respect to the orientation in $G_{xy}$.  Finally, the crossing data for the grid diagram $G_{xy}$ states that, at a crossing, the segment parallel to the $y$-axis has a greater $z$-coordinate than the segment parallel to the $x$-axis.  This matches the first  crossing conditions above.  The same facts are also true for the other two projections, $\pi_{yz}:\Gamma\ra G_{yz}$ and $\pi_{zx}:\Gamma \ra G_{zx}$.  Figure~\ref{goodex2} shows the different projections of a cube diagram to $G_{xy}, G_{yz}$ and $G_{zx}$.

\begin{figure}[H]
\includegraphics[scale=.3]{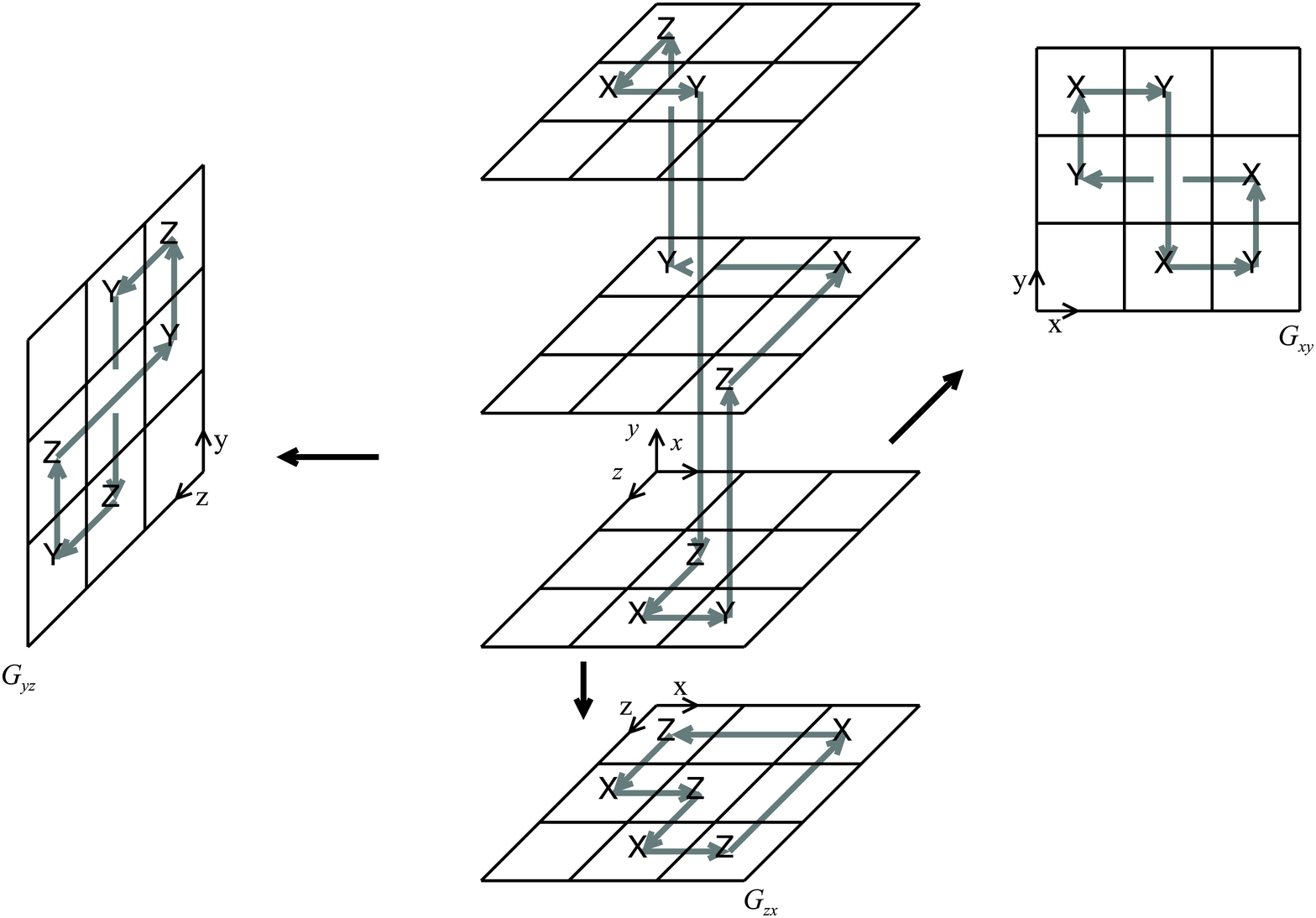}
\caption{\small \it A cube diagram and its three projections to oriented grid diagrams $G_{xy}$, $G_{yz}$, and $G_{zx}$. } \label{goodex2}
\end{figure}

\bigskip

\bigskip
\section{Hypercube Diagrams in 4-dimension}
\label{hypercubediagrams}
\bigskip

Hypercube diagrams in 4-dimensions are defined similarly to cube diagrams.  There is one important difference:  Given a hypercube in $\BR^4$ using a $wxyz$ coordinate system, the first thought one may have is to define a hypercube diagram structure by requiring each of the 4 projections $\pi_w,\pi_x,\pi_x,\pi_y:\BR^4\ra\BR^3$ be cube diagrams.  This use of projections turns out to be too strong of a condition---doing so forces all four cube diagrams to be cube diagrams of the same link.  Hypercube diagrams in $4$-dimensions should capture more information about tori than a link in $3$-dimensions.    The solution is to project the hypercube to $2$-planes instead.  Of course, the $2$-plane projections should be oriented grid diagrams.

\medskip

Recall from Baldridge-Lowrance \cite{BL} that only two grid diagram projections of a cube diagram are needed to reconstruct the original cube diagram.  This idea can be generalized to hypercubes in $4$ dimensions: there is a natural way to select two of the projections ${\pi_w,\pi_x,\pi_y,\pi_z:\BR^4\ra\BR^3}$ (where each map projects out the coordinate indicated) and two projections of each of those data structures to $2$-planes and require all four $2$-plane projections to be oriented grid diagrams (see Figure~\ref{projecthypercube}).  None of the projections from $\BR^4\ra \BR^3$ are required to be cube diagrams.  Also note in the picture below that, for example, one can also get a projection of the hypercube to the $zy$-plane by first projecting out the $x$-coordinate using $\pi_x$, and then projecting out the $w$-coordinate in $\BR^3$.  We do not require that projection of the hypercube (which is often different than $G_{yz}$) to be an oriented grid diagram.

\begin{figure}[H]
\includegraphics[scale=1]{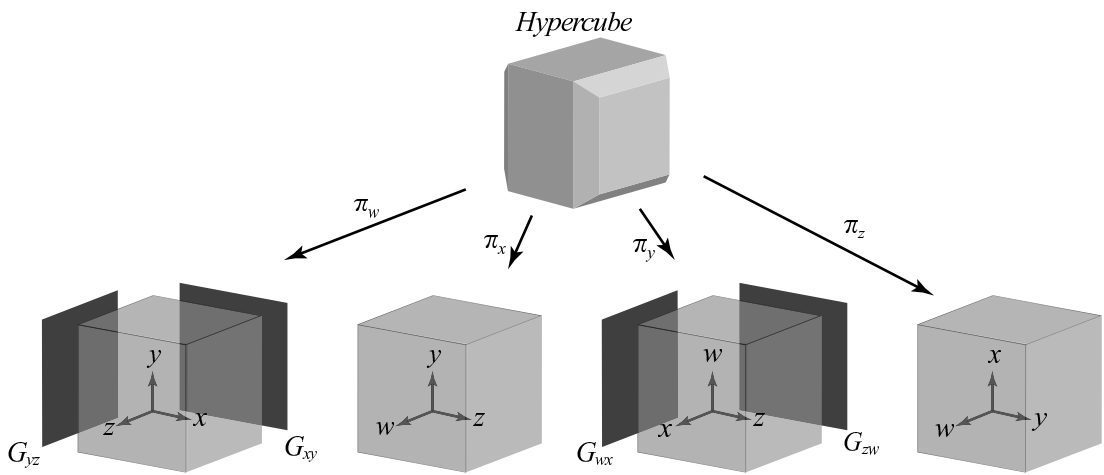}
\caption{\small \it The projections to the four oriented grid diagrams $G_{xy}$, $G_{yz}$, $G_{wx}$, and $G_{zw}$. } \label{projecthypercube}
\end{figure}

\medskip

The definition of a hypercube diagram codifies the discussion above with a data structure that mimics  the definition of a cube diagram.  Like a cube diagram, let $n$ be a positive integer and  let the hypercube $[0,n]\times [0,n]\times [0,n] \times [0,n] \subset \mathbb{R}^4$ be thought of as a $4$-dimensional Cartesian grid, i.e., a grid with integer valued vertices with axes $w$, $x$, $y$, and $z$.  Orient $\BR^4$ with the orientation $w\wedge x\wedge y\wedge z$.

\medskip

 A \textit{flat} is any right rectangular $4$-dimensional prism with integer valued vertices in the hypercube such that there are two orthogonal edges at a vertex of length $n$ and the remaining two  orthogonal edges are of length $1$.  Name flats by the axes parallel to the two orthogonal edges of length $n$.  For example, a $yz$-flat is a flat that has a face that is an $n\times n$ square that is parallel to the $yz$-plane.

\medskip

Similarly, a {\em cube} is any right rectangular $4$-dimensional prism with integer vertices in the hypercube such that there are three orthogonal edges of length $n$ at a vertex with the remaining orthogonal edge of  length $1$.  Name cubes by the three  edges of the cube of length $n$.  See Figure~\ref{cubesandflats} for examples.

\medskip

\begin{figure}[H]
\includegraphics[scale=1]{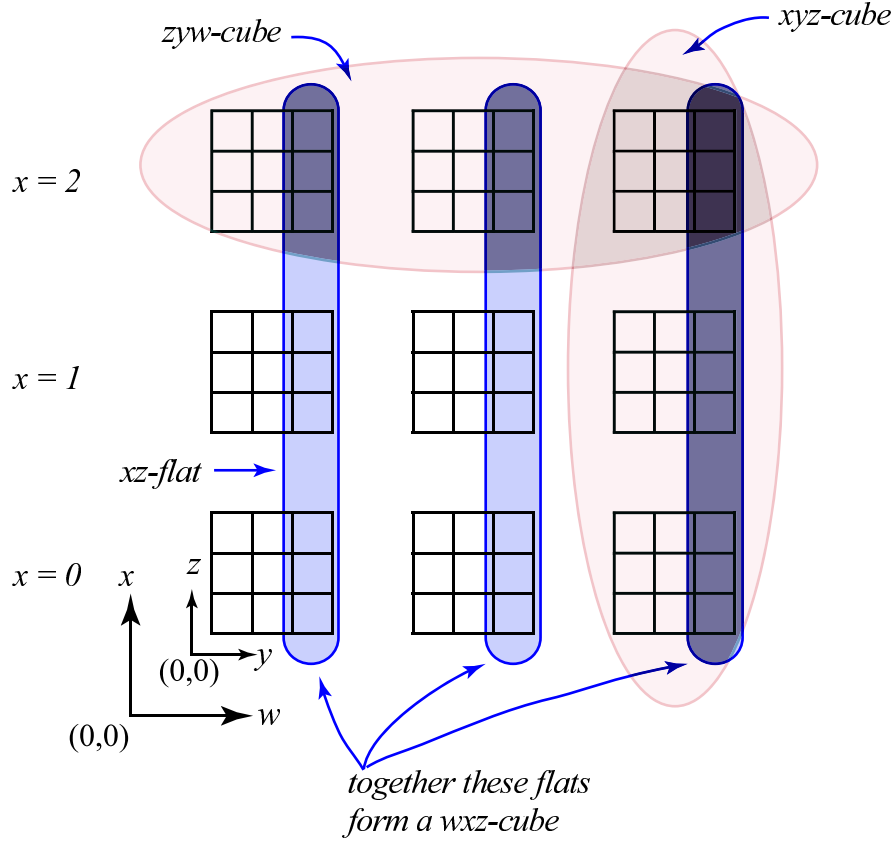}
\caption{\small \it  A schematic for displaying a hypercube diagram.  The outer $w$ and $x$ coordinates indicate the ``level'' of each $yz$-flat.  The inner $y$ and $z$ coordinates start at $(0,0)$ for each of the nine $yz$-flats. With these conventions understood, it is then easy to display $xz$-flats, $zyw$-cubes, $xyz$-cubes, $wxz$-cubes, etc. } \label{cubesandflats}
\end{figure}

\medskip

Like cube diagrams, a marking is a labeled point in $\BR^4$ with half-integer coordinates.  Mark unit hypercubes in the $4$-dimensional Cartesian grid with either a $W$, $X$, $Y$, or $Z$ such that the following {\em marking conditions} hold:

\begin{itemize}
 \item each cube has exactly one $W$, one $X$, one $Y$, and one $Z$ marking;\\

 \item each cube has exactly two flats containing exactly 3 markings in each; \\

 \item for each flat containing exactly 3 markings, the markings in that flat form a right angle such that each ray is parallel to a coordinate axis;\\

    \item for each flat containing exactly 3 markings, the marking that is the vertex of the right angle is $W$ if and only if the flat is a $zw$-flat, $X$ if and only if the flat is a $wx$-flat, $Y$ if and only if the flat is a $xy$-flat, and $Z$ if and only if the flat is a $yz$-flat.

\end{itemize}

\medskip

Observe that the 4th condition rules out the possibility of either $yw$-flats or a $zx$-flats with three markings.    This restriction will become important later when representing  tori  as hypercube diagrams.

\medskip

The conditions above on the markings ensure that markings can be connected together by segments that go from $W$ to $X$, from $X$ to $Y$, from $Y$ to $Z$, and from $Z$ to $W$.  Furthermore, these segments join together to form embedded loops in the $4$-dimensional cube.  Of course these loops are not ``linked'' in the literal sense of embedded $S^1$'s in $\BR^4$, but after the crossing conditions are imposed below, it will be clear that the embedded circles can't always be ``pulled apart'' using hypercube diagram moves (which are defined in the next section).

\medskip

\begin{definition} In a hypercube $H\Gamma$, call the markings together with the segments between the markings the {\em hyperlink}.
\end{definition}

\medskip

Generating a set of markings that satisfies the marking conditions is actually quite easy and quick. First, like oriented grid diagrams and cube diagrams from before, the segment that goes from an $X$ marking to a $Y$ marking is always parallel to the $x$-axis, so all ``vectors'' starting at the $X$ markings are parallel to the $x$-axis.  The same is true for the other segments: segments that start at $W$ markings are parallel to the $w$-axis, and so on.  With that understood, choosing a set of markings that satisfy the marking conditions is straightforward.  Start with an empty cube diagram schematic (see Figure~\ref{cubesandflats} above) and

\begin{enumerate}
\item Choose one $yz$-flat from each row of $yz$-flats so no two chosen $yz$-flats are in the same column.\\

\item Place one $W$ marking in each chosen $yz$-flat so that no two $W$ markings are in the same relative row or column in the $yz$-flats.\\

\item Choose an $X$ marking for a given $W$ marking by selecting a different column of $yz$-flats and putting the $X$ marking in the same $yz$-flat row as the $W$ marking  and in the same relative row and column as the $W$ marking ($W$ goes to $X$ parallel to the $w$-axis), making sure to have only one $X$ marking per column.\\

\item For each column, put a $Y$ marking in the same $yz$-flat as the $W$ marking and in the same relative position as the $X$ marking in the column.\\

\item In each $yz$-flat that has a $W$ and $Y$ marking, put a $Z$ marking in a square that is in the same column of the $Y$ marking and the same row of the $W$ marking.\\

\end{enumerate}

The schematic produced by the steps above has several $yz$-flats that contain $Y$, $Z$, and $W$  markings, and other $yz$-flats that contain a single $X$ marking.  Figure~\ref{trefoilhypercube} has  an example of a hypercube diagram schematic representing a standard torus on the left and an example of a schematic of a ``Trefoiled'' torus on the right.

\begin{figure}[H]
\includegraphics[scale=1]{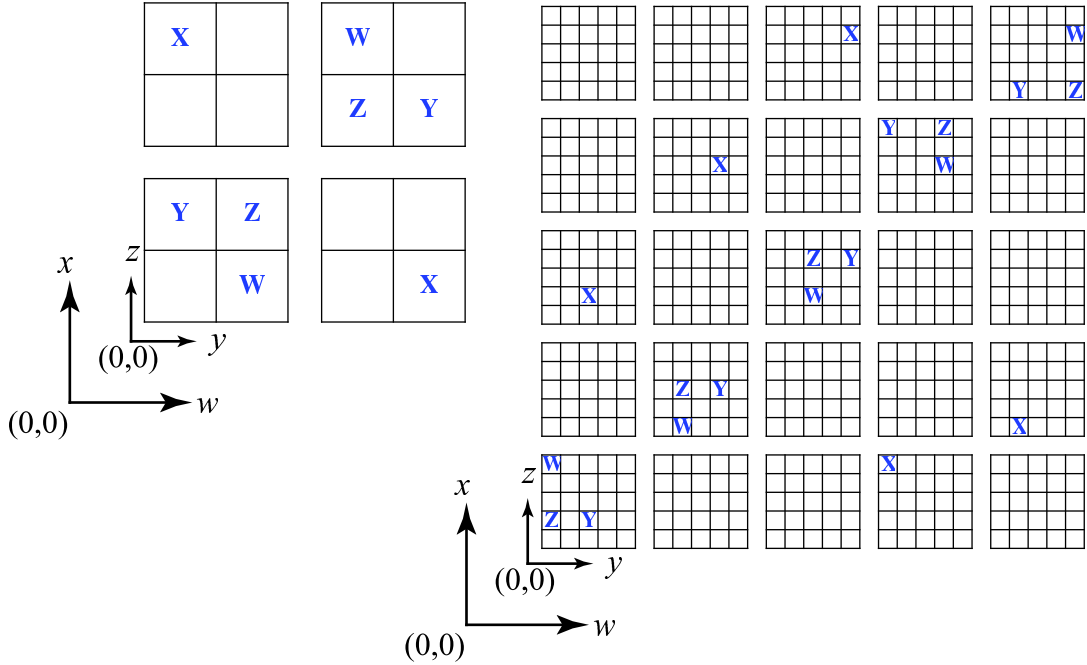}
\caption{\small \it  Examples of hypercubes.  On the left is the simplest hypercube possible (a torus).   On the right is a hypercube diagram that projects to a Trefoil grid diagram (seen by projecting the $yz$-flats down the diagonal). } \label{trefoilhypercube}
\end{figure}

\medskip

The crossing conditions for a hypercube are easy to describe but hard to check (without the aid of a computer).  First, consider two projections of the hypercube to $3$-space $\pi_w:\BR^4 \ra \BR^3$ and $\pi_y:\BR^4\ra\BR^3$, including the projections of the markings.  For the $\pi_w$ projection, each $W$ and $X$ marking related by a segment project to the same point, call that point an $X$ marking in $\BR^3$.  Similarly, the $\pi_y$ projection maps each $Y$ marking and $Z$ marking related by a segment to the same point, call it a $Z$ marking.  All remaining markings are mapped  to their own point in $\BR^3$ for each projection.  Name the resulting data structures (i.e., the image of the hypercube:   $[0,n]\times[0,n]\times[0,n]\subset \BR^3$ together with the projected markings) by $C_{xyz}$ for the $\pi_w$ projection and $C_{wxz}$ for the $\pi_y$ projection.  For example, the $\pi_w$ projection $C_{xyz}$ for the ``Trefoil'' example in Figure~\ref{trefoilhypercube}  can be found by shifting all the markings to the left-most column, remembering to replace the ``$W-X$'' markings with $X$ markings.

\medskip

We carefully described the definitions of oriented grid diagrams and cube diagrams to make the next the statement obvious:  The data structures $C_{xyz}$ and $C_{wxz}$ both satisfy the marking conditions of cube diagrams.  Recall that for grid diagrams and cube diagrams, the segment that begins at a marking is always parallel to the axis defined by the marking.    So in the case of $\pi_w:\BR^4\ra C_{xyz}$, the $W$ marking, $X$ marking, and the segment going from $W$ to $X$ all get projected to a single point, the $X$ marking.  The segment from the $X$ marking continues to begin at the new $X$ marking and is parallel to the $x$-axis in the projection.  Likewise, the remaining $Y$ and $Z$ markings and the segments that begin from them continue to remain parallel to axes defined by the markings in the projection.

\medskip

It is not true, in general, that a randomly chosen set of markings in the hypercube will project to data structures $C_{xyz}$ and $C_{wxz}$ that also satisfy the cube diagram crossing conditions (in fact, it is  exceedingly rare---cf. \cite{note} for calculations about the sparsity of cube diagrams).  Nor is it advantageous to require that the data structures $C_{xyz}$ and $C_{wxz}$ to be cube diagrams---since they both share a projection to the $zx$-plane, they would share a common grid diagram up to orientation, which means both data structures would essentially describe the same link.

\medskip

Of the six $2$-planes one can project a hypercube to, $wx$, $yz$, $zw$, $xy$, $yw$ and $zx$, two of the planes are ruled out by the hypercube marking conditions (cubes cannot have $yw$- or $zx$-flats with three markings).  This leaves four possible projections to $2$-planes.  We require all of these projections to be oriented grid diagrams.

\medskip

Given a data structure of markings in a hypercube, the data structure satisfies the {\em crossing conditions} if

\begin{itemize}
\item For the data structure $C_{xyz}$ given by the projection $\pi_w:\BR^4\ra\BR^3$, the images of the projections $\pi_x:C_{xyz}\ra G_{yz}$ and $\pi_z: C_{xyz} \ra G_{xy}$  are oriented grid diagrams.\\

\item For the data structure $C_{wxz}$ given by the projection $\pi_y:\BR^4\ra\BR^3$, the images of the projections $\pi_z:C_{wxz}\ra G_{wx}$ and $\pi_w:C_{wxy}\ra G_{xy}$ are oriented grid diagrams.
\end{itemize}

\medskip

If the hypercube and data structure satisfies both marking and crossing conditions, then it is called a {\it hypercube diagram} and denoted by $H\Gamma$.  We may also denote the hypercube diagram by $H\Gamma(L_1,L_2)$ where $L_1$ is the link given by the grid diagram $G_{wx}$ and $L_2$ is the link given by the grid diagram $G_{yz}$.   For the rest of this paper, we will often drop $C_{xyz}$ and $C_{wxz}$ and just refer to oriented grid diagrams $G_{wx}$, $G_{yz}$, $G_{xy}$, and $G_{zw}$ of $H\Gamma$ without reference to the cubes they are projections of.

\medskip

The reader may wonder why the projections $\pi_x,\pi_z:\BR^4\ra \BR^3$ were ignored.  Note that one could demand that, for example, $\pi_w\circ \pi_x:H\Gamma \ra G_{zy}$ be a grid diagram as well  (it is not true that $\pi_w\circ\pi_x$ and $\pi_x \circ \pi_w$ give rise to the same data structures, or even that both are grid diagrams when one of them is already known to be a grid diagram).  The reason the grid diagrams $G_{wx}$, $G_{yz}$, $G_{xy}$, and $G_{zw}$ were chosen is that we want a theory for {\em oriented} tori and these grid diagrams are compatible with the standard right hand orientation.  We could have set up an analogous theory instead using the oriented grid diagrams $G_{xw}$, $G_{zy}$, $G_{yx}$, and $G_{wz}$ for the data structures $C_{zyw}$ and $C_{yxw}$ (see Figure~\ref{projecthypercube}).

\bigskip

\bigskip
\section{Hypercube moves in 4-dimension}
\label{hypercubemoves}
\bigskip

There are three hypercube moves on hypercubes that take a valid hypercube to another valid hypercube (i.e., after each move, the projections to $G_{wx}$, $G_{yz}$, $G_{xy}$, and $G_{zw}$ are valid oriented grid diagrams).  To understand hypercube moves, we quickly review the two different types of moves on a grid diagram.

\subsection{Grid diagram moves}  Following Cromwell \cite{Cromwell} (cf. Dynnikov~\cite{Dynnikov}), any two grid diagrams for the same link can be connected by a sequence of the following  two elementary moves:

\bigskip

\noindent {\bf Stabilization.} A stabilization move adds an extra row and column to a grid diagram.  Suppose $G_{xy}$ is an oriented grid diagram with grid number $n$.  Let $X$ and $Y$ be two marking that form a line parallel to the $x$-axis.   Split the row they are in into two new rows, with $X$ in one new row and $Y$ in the other so that the $x$-coordinate of both markings remain the same as in $G_{xy}$.  Also add a new column to the diagram adjacent to the $X$ or the $Y$ marking  such that it is between the two markings, and then place two new markings $X'$ and $Y'$ in the column so that $X'$ occupies the same row as $Y$ and $Y'$ occupies the same row as $X$. See Figure \ref{stab}.  Destabilization is the inverse of stabilization.

\begin{figure}[H]
\includegraphics[scale=.27]{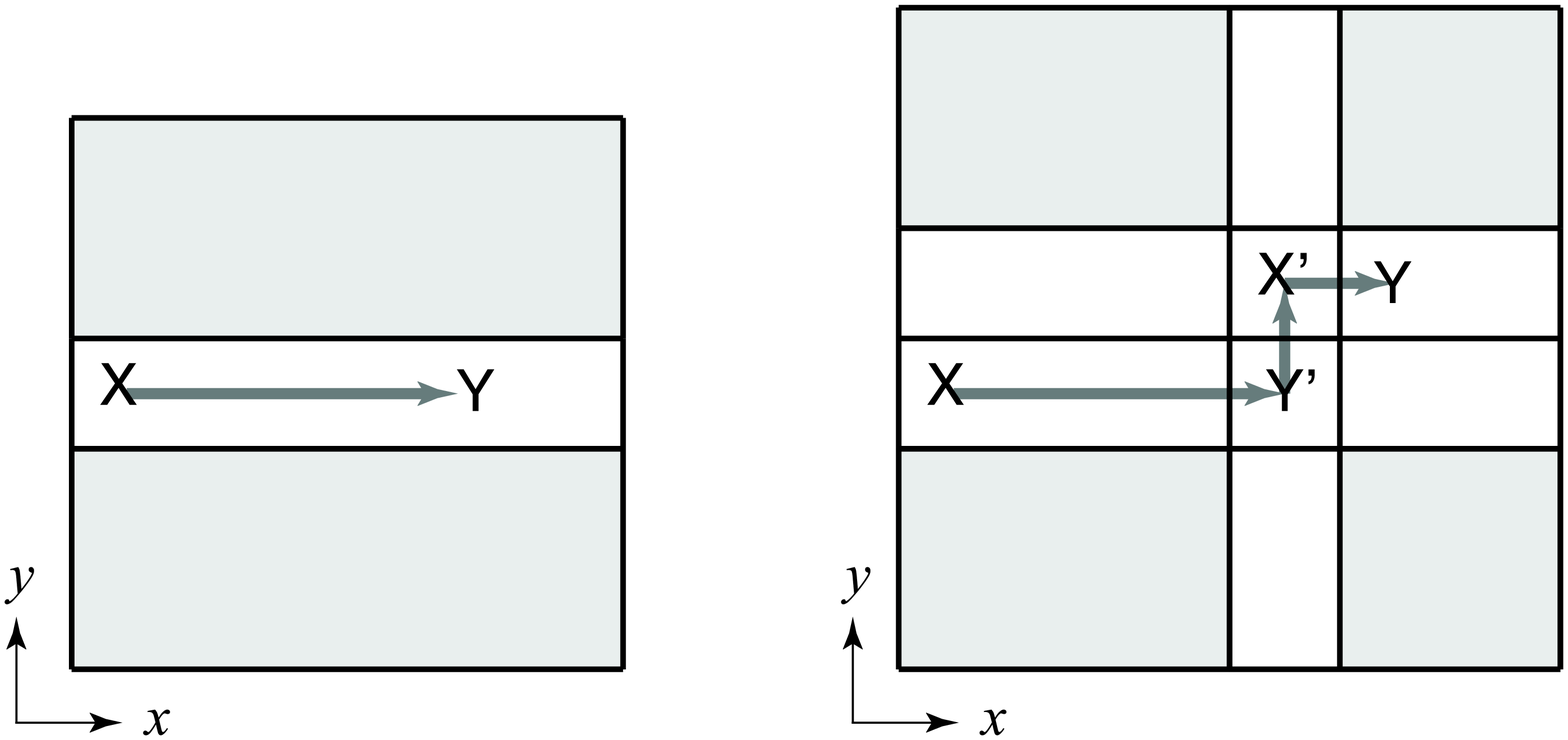}
\caption{A {\bf stabilization move} replaces the segment on the left by the three segments on the right, which increases the size of the grid diagram by one.} \label{stab}
\end{figure} 

\bigskip

\noindent {\bf Commutation.} Consider two adjacent rows in the grid diagram. In each row parallel to the $x$-axis, project the line segment connecting the $X$ and $Y$ to the $x$-axis. If the projections of the segments are disjoint, share exactly one point, or if the projection of one segment is entirely contained in the projection of the other, then the two rows can be interchanged. There is a similar move for columns.
\begin{figure}[H]
\includegraphics[scale=.27]{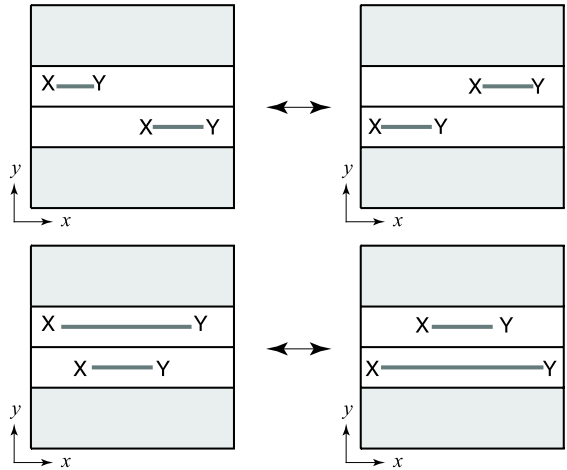}
\caption{A {\bf commutation move} interchanges adjacent rows (or columns) when the markings are situated as above.} \label{commutation}
\end{figure}

\medskip

In the literature there is also a cyclic permutation move.  The cyclic permutation move can be shown to be a consequence of the moves above (cf. \cite{BL}).  We definitely need the second commutation move instead of a cyclic permutation move for hypercubes---it can be shown, for example, that a cyclic permutation of a cube diagram can easily give cubes that do not satisfy the crossing conditions (after permuting, the resulting data and markings may not even represent the same knot type).  By using the second commutation move, we can keep our moves to the grid $G_{xy}$ and do not need to consider the grid to be on a torus, as it is often presented in the literature.

\subsection{Hypercube moves}  The first two moves on a $4$-dimensional hypercube $H\Gamma$ are just generalizations of the grid diagram moves and cube diagram moves. The third move, called a swap move, is new and due to the $4$-dimensional nature of the hypercube.  A problem with the commutation move is that, while the move is easy to describe abstractly, it can be difficult in practice to determine whether or not a given commutation move can be done or not.  In this section we will only describe the moves abstractly and provide a few examples to help the reader see that the moves make sense on a hypercube.  Future work in the same vein as in \cite{CS} needs to be done to give a few simple pictures like Figures~\ref{stab} and \ref{commutation} to determine when hypercube commutation can be done on schematics like those shown in Figure~\ref{trefoilhypercube}.

\bigskip

\noindent {\bf Hypercube Stabilization.}  A cube stabilization move increases the size of the hypercube diagram by 1.   The basic idea is to replace a marking, say $W$, with a new chain of markings $W\ra X\ra Y\ra Z \ra W$ where the lengths of all the segments in the chain are length 1.  The process of stabilizing is easy:  consider the segment leaving $W$ parallel to the $w$-axis.   If that segment is pointing in the positive direction, then insert  an $xyz$-cube, a $zyw$-cube, a $wxz$-cube and a $yxw$-cube next to the $W$ marking, each time on the side further away from the coordinate axes.  If the segment from $W$ points in the negative direction, insert the cubes on the side of the $W$ marking closer to the coordinate axes.    The intersection of these four cubes is a unit cube, place a new $W$ marking in that unit cube.  There is then a unique way to put new $X$, $Y$, and $Z$ markings into the new cube.  The result of this process is a stabilization of the hypercube $H\Gamma$.  Figure~\ref{hyperstab} shows what this process looks like in a hypercube schematic.  Figure~\ref{hyperstab2} shows how the torus is represented by the hypercube diagrams before and after a stabilization.  Destabilization is the inverse of this process.

\medskip

\begin{figure}[H]
\includegraphics[scale=1]{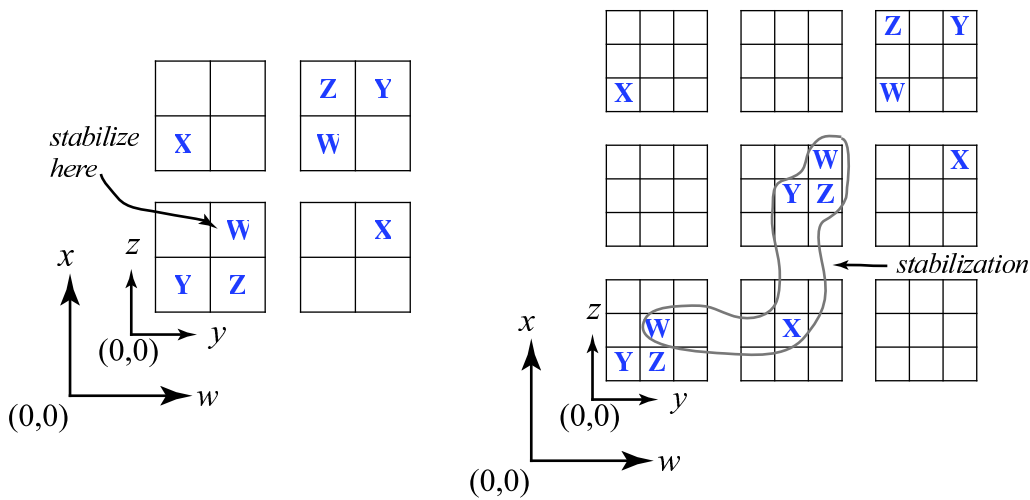}
\caption{An example of a {\bf hypercube stabilization move} at the $W$ marking in the picture on the left.  On the left is a hypercube diagram schematic of the resulting stabilization. } \label{hyperstab}
\end{figure}

A hypercube stabilization of $H\Gamma$ induces grid stabilizations in each of the four projections $G_{wx}$, $G_{yz}$, $G_{xy}$, and $G_{zw}$.  Therefore the links represented by $G_{wx}$ and $G_{yz}$ before a hypercube stabilization are the same links as after the stabilization.

\begin{figure}[H]
\includegraphics[scale=1]{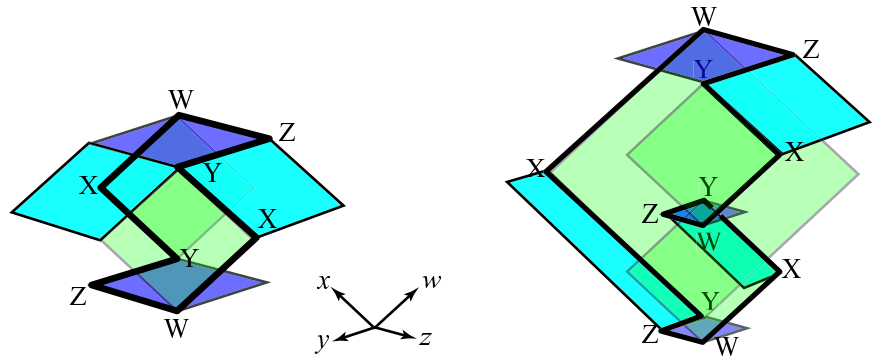}
\caption{A partial picture of the torus before and after the {\bf hypercube stabilization move} at the $W$ marking of the hypercube diagram schematic in Figure~\ref{hyperstab}.   See Figures~\ref{std_torus} and \ref{std_torus2} for the complete picture of the standard torus on the left. } \label{hyperstab2}
\end{figure}

\bigskip

\noindent {\bf Hypercube Commutation.}  A hypercube commutation move can be described as follows:  commute any two adjacent cubes in the hypercube and if the resulting data structure is also a hypercube, call that move a {\em hypercube commutation move}.  Figure~\ref{hypercommutation} shows examples of adjacent cubes that can be commuted in a hypercube.

\begin{figure}[H]
\includegraphics[scale=1]{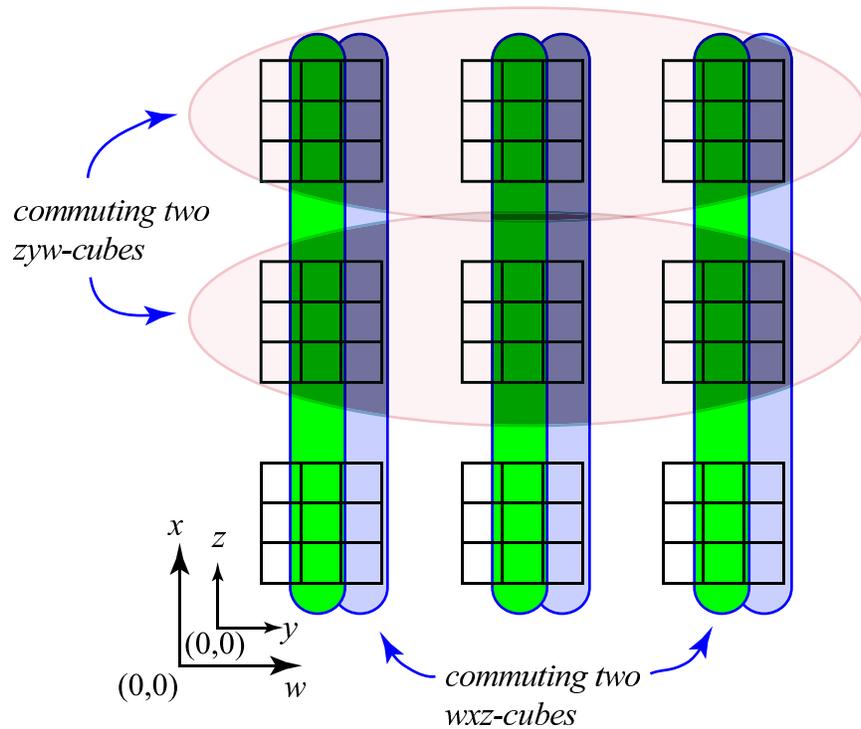}
\caption{Examples of ways to commute cubes in a hypercube.} \label{hypercommutation}
\end{figure}
\medskip

More specifically, a hypercube commutation of two adjacent cubes will change two of the four oriented grid diagrams projections $G_{wx}, G_{yz}, G_{xy}, G_{zw}$ by a grid commutation move.  There are two  configurations in Figure~\ref{commutation} for commutation moves in the grid.   Therefore we can think of hypercube commutation moves as only those moves that look like any combination of these two configurations in two of the projections $G_{wx}, G_{yz}, G_{xy}, G_{zw}$.  Note that  a hypercube commutation move really is a $4$-dimensional move---you cannot arbitrarily commute two adjacent columns parallel to the $x$-axis in $G_{wx}$ and two adjacent rows parallel to the $z$-axis in $G_{zw}$ and expect that the new set of four oriented grid diagram projections to even correspond to a data structure that satisfies the marking conditions.

\medskip

While the work on easily identifying hypercube commutation configurations has yet to be done, it is easy to see the results of a successful cube commutation.  For example, Figure~\ref{hypercommutation_ex} is the result of doing two hypercube commutation moves on the hypercube schematic on the right in Figure~\ref{hyperstab} above.  To get the hypercube diagram schematic below, first commute the first two rows of $zyw$-cubes and then commute the first two columns of $xyz$-cubes.  The picture next to the schematic is of the resulting torus.

\begin{figure}[H]
\includegraphics[scale=1]{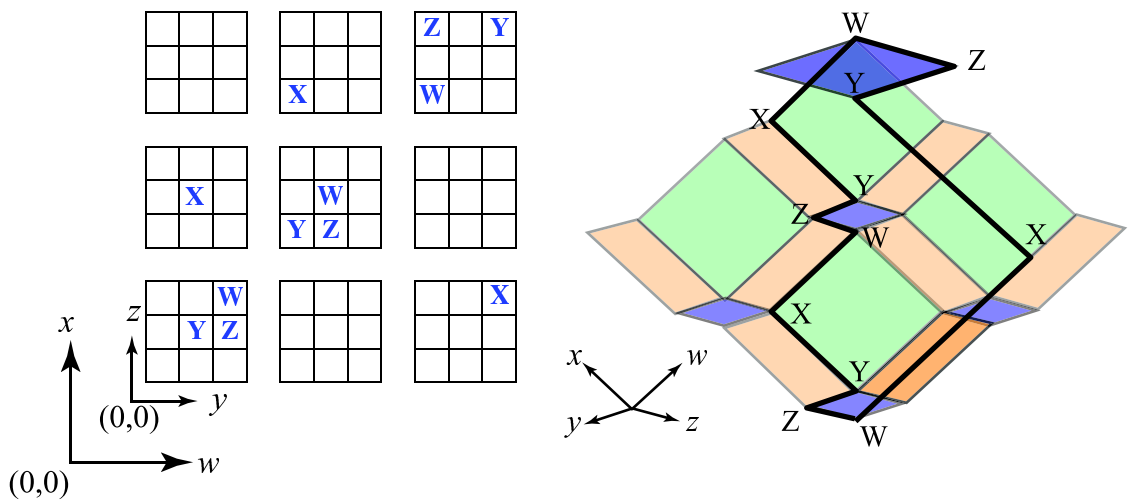}
\caption{The hypercube on the right of Figure~\ref{hyperstab} after two commutations.  To see this, first  commute the $x=0$ and $x=1$ $zyw$-cubes, then commute $w=0$ and $w=1$ $xyz$-cubes.  A {\em partial} picture of the resulting torus is on the right (see Figure~\ref{step4} for the complete picture). } \label{hypercommutation_ex}
\end{figure}
\medskip

Because a hypercube commutation move induces grid commutations in two of the four oriented grid projections $G_{wx}, G_{yz}, G_{xy}, G_{zw}$, the links associated to those projections remain unchanged.

\bigskip
\noindent {\bf Hypercube Swap.}  The simplest hypercube swap move switches the oriented grid diagrams $G_{wx}$ and $G_{yz}$ and switches the oriented grid diagrams $G_{xy}$ and $G_{zw}$.   The map that does this is simply $SW:\BR^4\ra\BR^4$ given by $SW(w,x,y,z)=(y,z,w,x)$ on the level of sets.  On the markings, $SW$ sends $W \ra Y$ and $X\ra Z$ and vice versa.  Figure~\ref{hyperswap} shows an example of a swap.

\begin{figure}[H]
\includegraphics[scale=1]{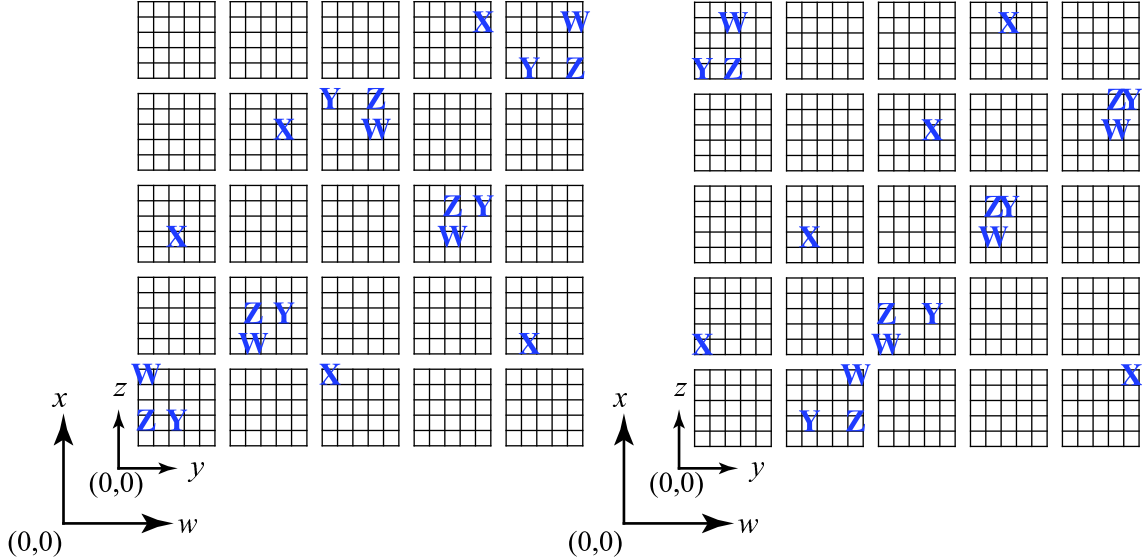}
\caption{The hypercube on the right is a swap of the hypercube on the left.} \label{hyperswap}
\end{figure}
\medskip

A swap move also includes swapping components of split links in the four grid diagram projections.  For example, let $H\Gamma$ be a hypercube diagram of size $n$ with two components that are disjoint from each other (two disjoint tori).  After possibly some commutation and stabilization moves, assume that both components project to split links in all four grid diagram projections $G_{wx}$, $G_{yz}$, $G_{xy}$, and $G_{zw}$ in block form.  Here block form means that the first component is always in the $[0,a]\times [0,a]$ region of all four grid diagrams and that second component is always in the $[a,n]\times[a,n]$ region of all four grid diagrams.  In this situation, we can swap either component simply by using the map defined above but only in the region of the component being swap.

\medskip

Because the hypercube swap move exchanges the $G_{wx}$ and $G_{yz}$ grid diagrams, it exchanges the two links that $G_{wx}$ and $G_{yz}$ represent.  However, the {\em pair} of links remain unchanged.

\bigskip

\begin{definition}
Any property of a hypercube that is invariant under the three moves defined above is called a {\em hypercube invariant}.
\end{definition}

Because the three moves preserve the pair $(L_1,L_2)$ associated to a hypercube diagram $H\Gamma$ up to swapping paired components, we get the following theorem:

\begin{theorem}
Let $H\Gamma$ be a hypercube diagram and let $L_1$ and $L_2$ be the links associated to the oriented grid diagrams $G_{wx}$ and $G_{yz}$.  Any hypercube invariant is also an invariant of the pair of links $(L_1,L_2)$ up to swapping paired components from $L_1$ to $L_2$ and vice versa.  If $L_1$ is a knot (and therefore so is $L_2$), then  any hypercube invariant is an invariant of the pair of knots.
\end{theorem}

\bigskip

\bigskip
\section{Algorithm for representing tori from hypercubes}
\label{Algorithm}
\bigskip

The construction of a torus from a hypercube diagram proceeds in steps.  Before we describe the procedure, it is important to point out the relationship between our earlier choices and the consequence they have in constructing a torus.

\medskip

Let $H\Gamma$ be a hypercube.  Recall that $4$ projections were chosen to be grid diagrams:
$$\pi_x\circ\pi_w:H\Gamma \ra G_{yz}, \ \ \pi_z\circ\pi_w:H\Gamma \ra G_{xy},\  \ \pi_x\circ\pi_y:H\Gamma\ra G_{zw},\ \  \pi_z\circ\pi_y:H\Gamma\ra G_{wx}.$$
Consider two adjoining segments of the hyperlink that connect a $W$ marking to an $X$ marking and an $X$ marking to a $Y$ marking (the segments $\overline{WX}$ and $\overline{XY}$).  Then the sequence of the projections $\pi_x\circ\pi_w$  ``sweeps'' out a rectangle with sides $\overline{WX}$ and $\overline{XY}$.

\medskip

Following this insight, we demand that all polygonal regions of the torus be in one of the four planes: $wx$-plane, $yz$-plane, $xy$-plane or $zw$-plane.  We also require that all such polygonal regions be rectangles.  These rectangles can be named as follows:   Recall that at the $W$ marking, the segment starting at $W$ is parallel to the $w$-axis.  Specify the $\overline{WX}$ segment by $W$ and the $\overline{XY}$ segment by $X$ and the rectangle formed from these two segments by $WX$.
\bigskip

\bigskip

\noindent{\bf Algorithm for building a torus from a hypercube}

\bigskip

In the algorithm given below, we include pictures of each step for the hypercube diagram shown in Figure~\ref{hypercommutation_ex}.

\begin{enumerate}

\item[\bf Step 1.]  Attach all $WX$ and $YZ$ rectangles to the hyperlink (Figure~\ref{step1}):

\begin{figure}[H]
\includegraphics[scale=1]{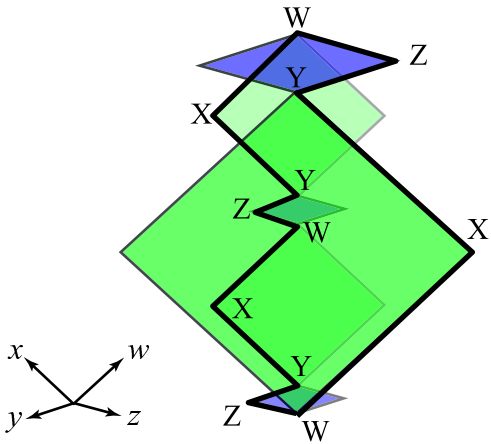}
\caption{Attach $WX$ and $YZ$ rectangles to the hyperlink.} \label{step1}
\end{figure}
\medskip

Note that the hyperlink {\em orients} these rectangles!  The orientation agrees with the projections to the grid diagrams.   For example, $WX$ rectangles  project to rectangles in $G_{wx}$ whose orientation agrees with the orientation given by the segments $WX$ and $XY$ (positive if counterclockwise, negative if clockwise). \\ \mbox{}\\

\item[\bf Step 2.] After attaching the $WX$ and $YZ$ rectangles to the hyperlink, each $W$ marking (and $Y$ marking) of the hyperlink has four segments attached to it coming from the boundary edges of the $WX$ and $YZ$ rectangles.  There is exactly one way to attach two rectangles to those four segments that is consistent with the orientations on the $WX$ and $YZ$ rectangles:  a $ZW$ rectangle attached to the $ZW$ and $WX$ segments, and an $XY$ rectangle attached to the other two segments (a rectangle on the ``opposite side'' of the attached $ZW$ rectangle).  Note that the orientation of the boundary given by the $WX$ and $YZ$ rectangles is opposite the orientation of the attaching rectangles themselves.

 \begin{figure}[H]
\includegraphics[scale=.9]{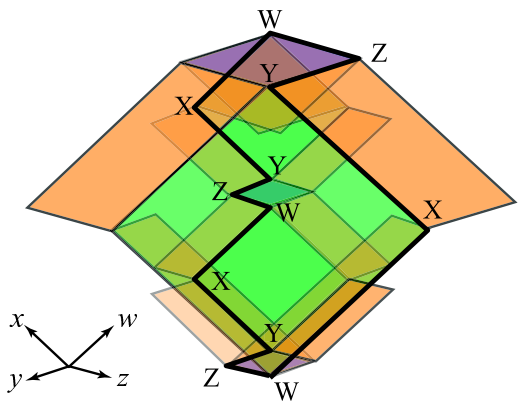}
\caption{Attach $XY$ and $ZW$ rectangles to the hyperlink.} \label{step2}
\end{figure}
\medskip

\item[\bf Step 3.] Continue the same procedure as in Step 2 at each new vertex that has two rectangles incident with it:   attach rectangles to the figure where at least two segments form a rectangle in the $wx$, $yz$, $xy$, or $zw$ planes  (no rectangles in the $yw$-planes or $zx$-planes).  There is  only one way to attach each rectangle consistent with the orientation of the rectangles incident with the rectangle that is being attached.  \\ \mbox{}\\

\begin{figure}[H]
\includegraphics[scale=.9]{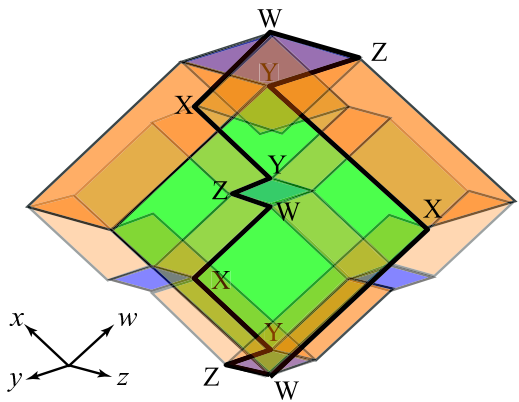}
\caption{Attach $WX$, $YZ$, $XY$ and $ZW$ rectangles between every two adjacent edges that form a rectangle in the $wx$, $yz$, $xy$, and $zw$ planes.  The picture above is only a partial picture of all the rectangles.  } \label{step3}
\end{figure}
\medskip

The reason that Figure~\ref{step3} is only a partial picture is that the full picture is hard to visualize.  The picture below shows how to insert the remaining rectangles.   The ``cap'' has been lifted off the surface to make it easier to see how the two disks are attached.\\

\begin{figure}[H]
\includegraphics[scale=.9]{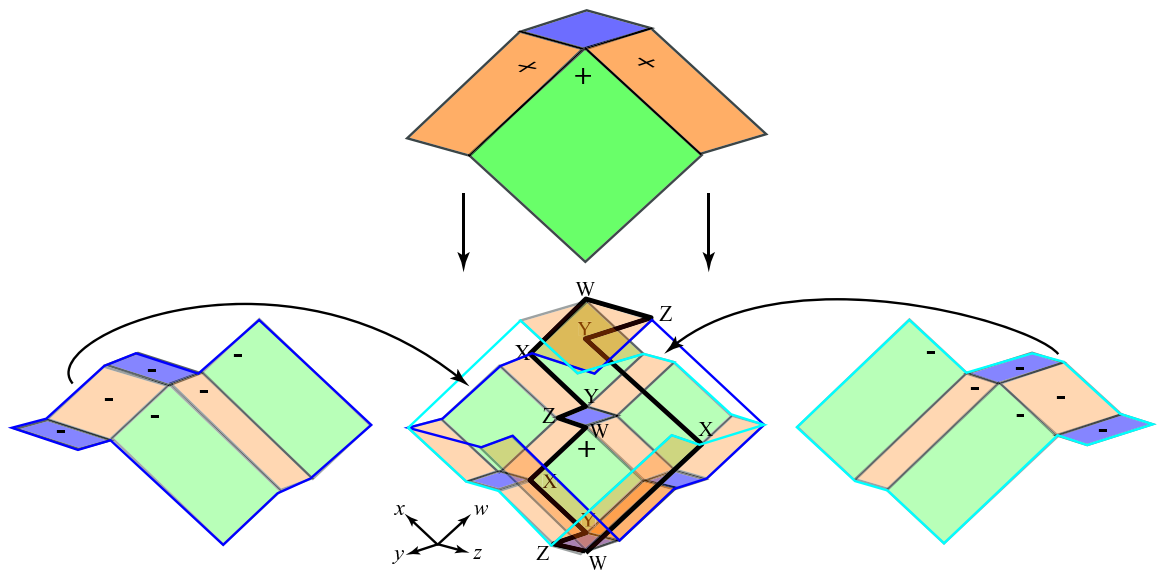}
\caption{The oriented torus determined by the hypercube diagram in Figure~\ref{hypercommutation_ex}.  } \label{step4}
\end{figure}
\medskip

\item[\bf Step 4.] We finish the algorithm with a description of the singularities of this PL-torus and when the PL-torus can be perturbed to an embedded torus in $\BR^4$.  There are four and only four naturally occurring singularities, of which the first two can be easily smoothed:

\begin{figure}[H]
\includegraphics[scale=.5]{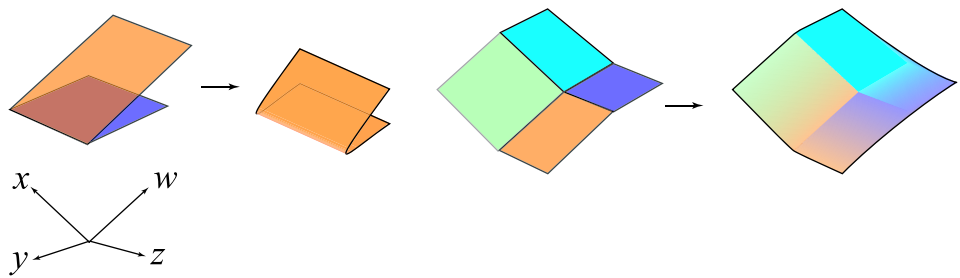}
\caption{Smoothing along the edge of two rectangles and at vertex.  } \label{smooth1}
\end{figure}
\medskip

The next type of singularity is when a rectangle passes through another rectangle.  Due to the marking conditions on the hypercube, such intersections  never intersect along an edge of either rectangle.  For example, suppose that a $WX$ rectangle passes through a $ZW$ rectangle as in the first picture below.  Then the $WX$ rectangle can perturbed positively in the $y$-direction as in the second picture and then smoothed as in the picture above.

\begin{figure}[H]
\includegraphics[scale=1]{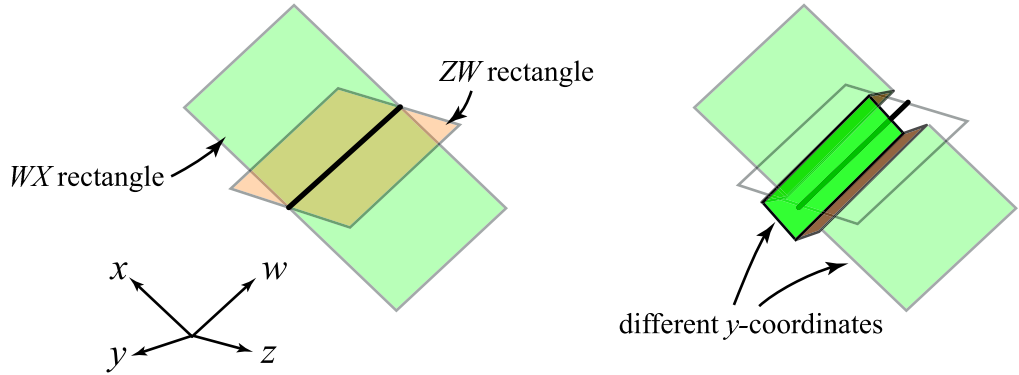}
\caption{Perturbing away the intersection of two rectangles that pass through each other. } \label{smooth2}
\end{figure}
\medskip

When two rectangles pass through each other, they share a common coordinate, and the intersection is never transverse in $\BR^4$.  Therefore there is always a coordinate in which one of the rectangles can be perturbed along to remove the intersection.

\medskip

As we will see in the proof below, if two rectangles pass through each other, then the intersecting segment is part of a circle called a {\em double point circle}.  If this circle doesn't intersect another double point circle, then the entire double point circle can be perturbed away.  For example, the segment in Figure~\ref{smooth2} is part of a circle  that is the union of segments made out of intersections of $WX$ and $ZW$ rectangles and intersections of $XY$ and $YZ$ rectangles.  By perturbing the $XY$ rectangle in the positive $w$-direction, the perturbation in Figure~\ref{smooth2} can be continued, as in the picture below:

\begin{figure}[H]
\includegraphics[scale=1]{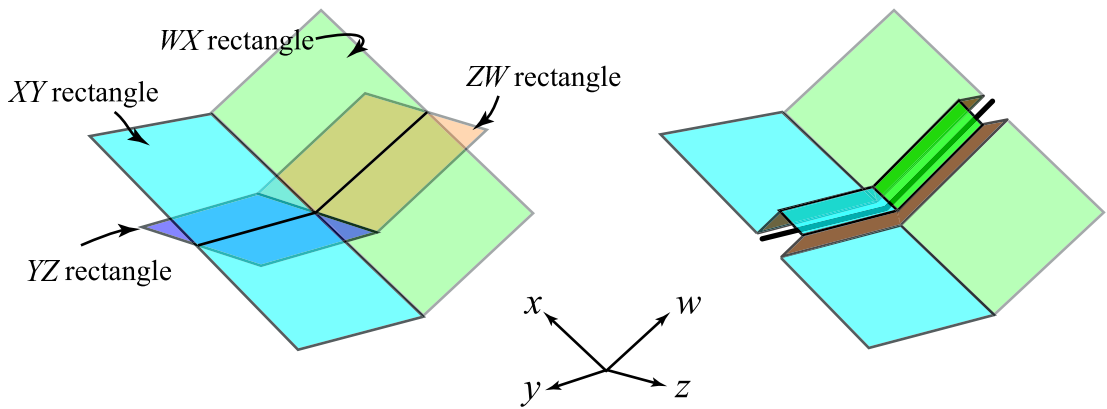}
\caption{Perturbing away the double point circle. } \label{smooth3}
\end{figure}
\medskip

Each rectangle comes with an orientation inherited by the original orientation on the hyperlink.  This orientation can be used to choose the the perturbation direction in a standard way for all rectangles that intersect.

\medskip

The fourth type of singularity cannot be perturbed away: when three rectangles intersect in the following configuration:

\begin{figure}[H]
\includegraphics[scale=1]{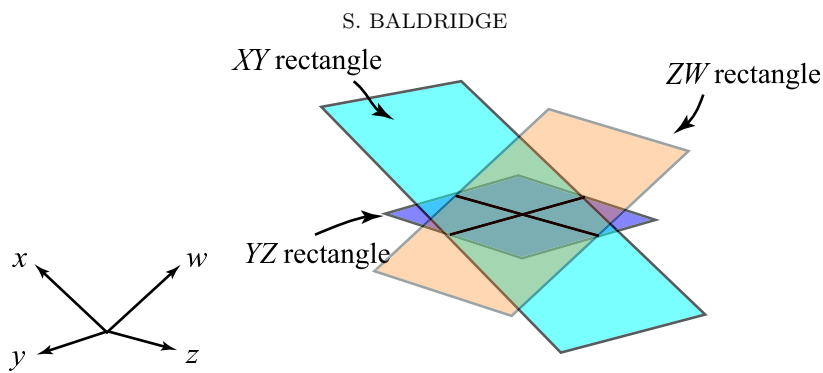}
\caption{A singularity that cannot be perturbed away. } \label{smooth4}
\end{figure}
\medskip

This type of singularity can be lifted off of the $YZ$ rectangle using the procedure above, but the perturbations of the $XY$ and $ZW$ rectangles will continue to intersect each other transversely in a point.
\end{enumerate}
 
 \bigskip
 
 We are now ready to prove:
 
 \medskip
 
 \begin{theorem}
Let $H\Gamma$ be a hypercube diagram.  Then $H\Gamma$ represents a $PL$-torus  which can be smoothed to an immersed Lagrangian torus in $\BR^4$ with regards to the symplectic form ${dw\wedge dy + dz\wedge dz}$.  If none of the double point circles intersect each other in the $PL$-torus, then the $PL$-torus can be smoothed to an embedded torus in $\BR^4$.  If there are no double point circles, then the $PL$-torus can be smoothed to an embedded Lagrangian torus in $\BR^4$.
\end{theorem}

Given the algorithm above, the proof has to deal with the following issues: Why does the insertion of rectangles into a hypercube give a closed surface?  Why is that surface a torus?  Why do rectangles that pass through each other do so only along double point circles?  When do double point circles intersect each other? When can the corners of the $PL$-torus be smoothed to get a Lagrangian torus?

\begin{proof}
Given a hypercube $H\Gamma$, consider its hypercube diagram schematic (we will work using the example in Figure~\ref{torusproof1} to guide the discussion).  Following the first step in the algorithm above, attach all $YZ$ rectangles into the schematic that are attach to the hyperlink, as in the picture below.

\medskip

\begin{figure}[H]
\includegraphics[scale=1]{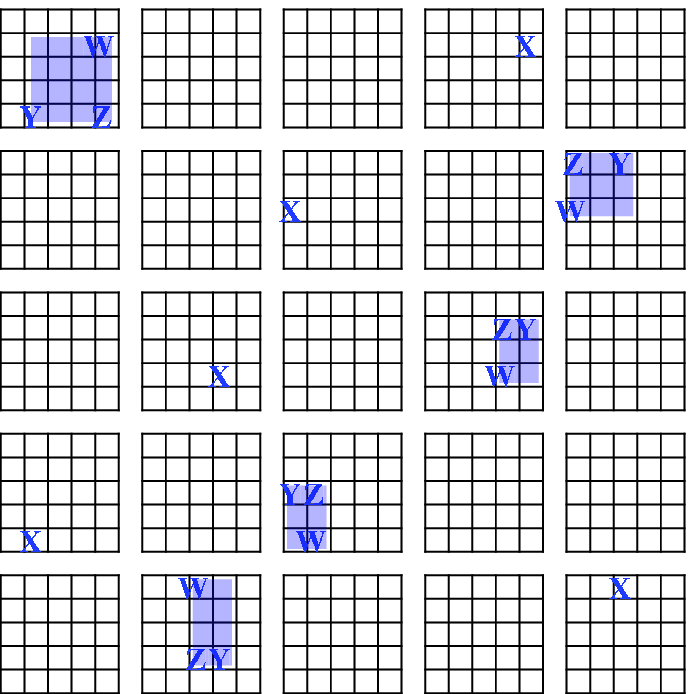}
\caption{A cube diagram schematic with $YZ$ rectangles filled in. } \label{torusproof1}
\end{figure}

Next, put rectangles into each $YZ$ flat in the following way:  for each $YZ$ flat, insert a rectangle  that has the same width and $y$-coordinate position  as the rectangle in that column and same height and $z$-coordinate position as the rectangle in that row.   See Figure~\ref{torusproof2}.

\begin{figure}[H]
\includegraphics[scale=1]{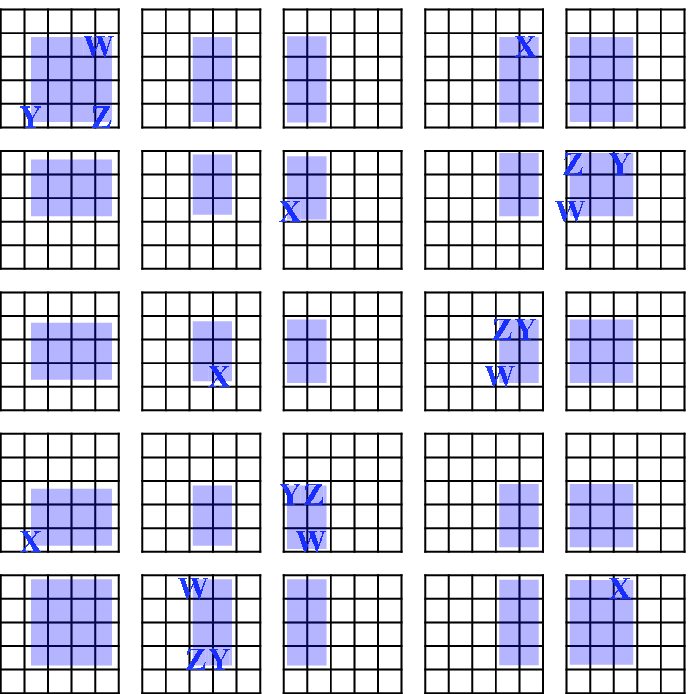}
\caption{A cube diagram schematic with $YZ$ rectangles filled in. } \label{torusproof2}
\end{figure}
\medskip

Connect $YZ$ rectangles by attaching $ZW$ rectangles to pairs of $YZ$ rectangles that are in the same row and have an edge in the same $y$-coordinate.  The hypercube diagram structure guarantees that there will be exactly $n$ distinct $ZW$ rectangles in each row (in each row, there are exactly two rectangles that have the same $y$-coordinate on one of their edges for each $y$-coordinate).

\medskip

Do the same for $XY$ rectangles: connect $YZ$ rectangles in the same column together by attaching $XY$ rectangles to pairs of $YZ$ rectangles that are in the same column and have an edge in the same $z$-coordinate.  Again, there will be exactly $n$ distinct $XY$ rectangles per column and $n^2$ $XY$ rectangles in the entire hypercube.

\medskip

At each of the four corners of each $YZ$ rectangle, there is exactly one $WX$ rectangle that can be inserted such that its edges agree with the $ZW$ and $XY$ rectangles that are incident to it.  Again, there are $n^2$ number of $WX$ rectangles, one unique rectangle for each $YZ$ flat.  Figure~\ref{torusproof3} gives some examples of $ZW$, $XY$, and $WX$ rectangles attached to the $YZ$ rectangles.

\begin{figure}[H]
\includegraphics[scale=1]{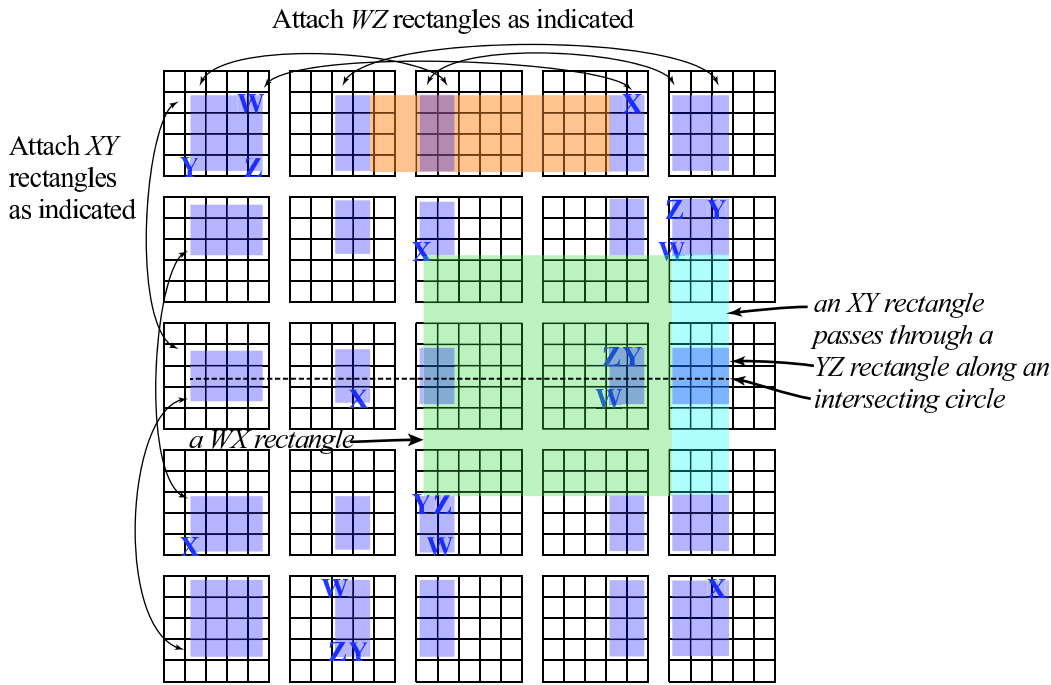}
\caption{A cube diagram schematic with examples of $ZW$, $XY$, and $WX$ rectangles. } \label{torusproof3}
\end{figure}
\medskip

A careful count of the edges used in this construction gives $8n^2$ edges ($4$ edges for each $YZ$ rectangle, $2$ new edges for each $ZW$ and $XY$ rectangle, and no new edges for each $WX$ rectangle).  Each edge intersects exactly two faces for a total of $4n^2$ faces.  Each vertex is incident to exactly 4 edges and 4 faces.  Altogether, there are four vertices in each $YZ$ flat for a total of $4n^2$ vertices.  Thus the union of the rectangles is clearly a closed (oriented) $PL$ surface with Euler characteristic $\chi = 4n^2-8n^2+4n^2=0$, i.e., a torus.

\medskip

Study the double point circle in Figure~\ref{torusproof3}.  Each $XY$ rectangle that goes from the 2nd row to the 4th row passes through a $YZ$ rectangle in row 3.  Similarly, each $WX$ rectangle that goes from the 2nd row to the 4th row passes through a $ZW$ rectangle in row 3 (the union of those $XY$ and $WX$ rectangles is a band around one of the homology generators of the torus).

\medskip

A symmetric argument shows that double point circles can form when $ZW$ and $WX$ rectangles pass through $YZ$ and $XY$ rectangles respectively.  Because of the marking conditions imposed on the hypercube, double point circles can only appear as disjoint horizontal lines or disjoint vertical lines on the hypercube diagram schematic, but horizontal and vertical double point circles can intersect.  Thus two double point circles intersect if and only if one of them is a horizontal circle and one of them is a vertical circle in a hypercube diagram schematic.

\medskip

If none of the double point circles intersect each other, then the union of double point circles can be smoothed to get an embedded torus as described in the algorithm.  If a horizontal and vertical double point circle meet, then the torus is immersed.

\medskip

\end{proof}

We finish the proof by showing that the smoothing of the $PL$-torus constructed in the proof above can be done in such a way that it is Lagrangian:

\medskip

\begin{lemma}
The $PL$-torus can be smoothed to an immersed Lagrangian torus in $\BR^4$ with respect to the symplectic form $\omega=dw\wedge dy + dz\wedge dx$.   If there are no double point circles in the $PL$-torus, then the Lagrangian torus is embedded in $\BR^4$.

\end{lemma}

\begin{proof}

First, all of the rectangles used to build the torus are in the $wx$, $xy$, $yz$, and $zw$ planes, which are Lagrangian planes with respect to the symplectic form defined by $\omega=dw\wedge dy + dz\wedge dx$.  Hence the $PL$-torus is already a ``Lagrangian'' torus in the $PL$-category.  We need only show that the edges and vertices where the rectangles adjoin can be rounded so that the resulting smoothing is still Lagrangian.  Without loss of generality, we describe equations for a $C^1$-immersed torus, noting that $C^\infty$ smoothing equations are also easy to write down but make the equations used in the proof  unnecessarily more complicated.

\medskip

Given $0<\varepsilon <<1$, we remove a closed $\varepsilon$-neighborhood from the edges and at the vertices of the rectangles used to construct the $PL$-torus and replace those neighborhoods as follows.   If the length of an edge joining a $WX$ rectangle at the origin to a $ZW$ rectangle at the origin  is of length $A\in \BZ$, replace the closed $\varepsilon$-neighborhood

$$[\varepsilon, A-\varepsilon] \times [0, \varepsilon] \times \{ 0 \} \times \{0\} \ \bigcup \ [\varepsilon, A-\varepsilon] \times \{0\} \times \{0\} \times [0,\varepsilon]$$

by the image of the map $E:[\varepsilon,A-\varepsilon]\times[0,\varepsilon \pi/2] \ra \BR^4$ given by

$$E(s,t) = (s, \varepsilon-\varepsilon \cos(t/\varepsilon), 0, \varepsilon-\varepsilon \sin(t/\varepsilon)).$$

A basis for the tangent vectors at a point $E(s,t)$ is then $$\langle \frac{\partial}{\partial w}, \ \sin(t/\varepsilon)\frac{\partial}{\partial x} - \cos(t/\varepsilon)\frac{\partial}{\partial z}\rangle.$$
Clearly, $\omega$ is zero on any two vector fields that are linear combinations of these basis elements.

\medskip

Similarly, we can write down equations at a corner.   Given a vertex at the origin, replace the closed $\varepsilon$-neighborhood of the original $PL$-torus

$$[0,\varepsilon] \times [0,\varepsilon] \times \{0\} \times \{0\}  \bigcup  \{0 \} \times [0,\varepsilon] \times [0,\varepsilon] \times \{0\} \bigcup \{0 \} \times  \{0 \} \times [0,\varepsilon] \times [0,\varepsilon] \bigcup [0,\varepsilon]  \{0 \} \times  \{0 \} \times [0,\varepsilon] $$

by the image of the map $V: [0,\varepsilon \pi/2] \times [0,\varepsilon \pi/2] \ra \BR^4$ given by

$$V(s,t) =(\varepsilon-\varepsilon \cos(s/\varepsilon), \varepsilon-\varepsilon \cos(t/\varepsilon), \varepsilon-\varepsilon \sin(s/\varepsilon), \varepsilon-\varepsilon \sin(t/\varepsilon))$$

Note that the image of the map $V$ matches up with the image of the map $E$ on the adjoining boundaries.

\medskip

A basis for the tangent vectors at a point $V(s,t)$ is
$$\langle  \sin(s/\varepsilon)\frac{\partial}{\partial w} - \cos(s/\varepsilon)\frac{\partial}{\partial y}, \ \sin(t/\varepsilon)\frac{\partial}{\partial x} - \cos(t/\varepsilon)\frac{\partial}{\partial z}\rangle.$$

For the same reasons as in the calculation above, $\omega$ is zero on any two vector fields that are linear combinations of these basis elements.

\medskip

Using the $C^\infty$-version of these replacements along each edge and at each vertex gives an immersed Lagrangian torus in $\BR^4$.  If there are no double point circles in the original $PL$-torus, then this Lagrangian torus is clearly embedded in $\BR^4$.  If the $PL$-torus has a double point circle, then Lagrangian torus remains immersed:  as Figure~\ref{smooth3} shows, perturbing the double point circle  away requires replacing parts of the $PL$-torus with surfaces where the symplectic form does not vanish.
\end{proof}

\bigskip
\section{Hypercube homology}
\label{hypercubehomology}
\bigskip

\bigskip

Given a hypercube diagram $H\Gamma$ and two of the four oriented grid diagram projections $G_{wx}$ and $G_{yz}$, we associate to $H\Gamma$ a bigraded chain complex. The generators of this complex depend on the choice of these two oriented grid diagram projections, but note that $G_{wx}$ and $G_{zw}$ are projections of the same link and $G_{yz}$ and $G_{xy}$ are also.  Therefore nothing is lost  by working with the $G_{wx}$ and $G_{yz}$ projections.

\medskip

The task in this section is simple and should be clear from how hypercube diagrams were defined: we will generalize combinatorial Knot Floer Homology to the hypercube diagram setting.  First, we describe Knot Floer Homology.

\bigskip
\subsection{Knot Floer Homology from Grid Diagrams}

In \cite{mos} and \cite{most}, knot Floer homology was computed using grid diagrams. This beautiful  construction encodes a Heegaard diagram as a grid diagram, recognizes generators of the chain complex of the homology as states on the grid diagram, and counts pseudo-holomorphic disks needed to define the differential  by counting rectangles between states in the grid diagram.

\medskip
We define knot Floer homology for a $G_{xy}$ grid diagram.  It directly corresponds to how grid diagrams are defined in the literature, but it follows from the definitions for hypercube and cube diagrams that the same definition works for $G_{wx}$ or $G_{yz}$ grid diagrams as well.  In subsequent sections we will use the notation developed below with obvious letter changes for $G_{wx}$ and $G_{yz}$ grid diagrams.

\medskip

Let $G_{xy}$ be an oriented grid diagram of size $n$ using the right hand orientation $x\wedge y\wedge z$ in $\BR^3$.    Each grid diagram $G_{xy}$ has an associated chain complex $(C^-(G_{xy}),\partial^-)$.   The chain complex $C^-(G _{xy})$ is generated by the set of states $S_{xy}$.  A state ${\bf s}$ is a set of $n$ lattice points of $G_{xy}$ with no coordinate equal to $n$ such that no two points of ${\bf s}$ determine a line parallel to the $x$-axis or the $y$-axis.  Essentially, ${\bf s}$ gives a one-to-one correspondence between the first $n$ lines in the grid parallel to the $x$-axis to the first $n$ lines in the grid parallel to the $y$-axis.

\medskip

The complex $C^-(G_{xy})$ has a Maslov grading and an Alexander filtration defined by maps from $S_{xy}$ to the integers or the half integers. To define the functions, we first need to define a way to count points of two sets.   Let $A$ and $B$ be collections of a finitely many points in the plane.  Define $I(A,B)$ to be the number of pairs $(a_1,a_2)\in A$ and $(b_1,b_2)\in B$ with $a_1<b_1$ and $a_2<b_2$. Then define $J(A,B)=(I(A,B)+I(B,A))/2$.    A state ${\bf s}\in S_{xy}$ is a collection of points with integer coordinates.  Also, the $X$ markings can be thought of as a set $\mathbb{X}_{xy} = \{X_i\}_{i=1}^n$ of points in the plane with half-integer coordinates.  We can apply $J$ to both sets.

\medskip

The Maslov grading is defined as
$$M_{xy}({\bf s})=J({\bf s}-\mathbb{X}_{xy},{\bf s}-\mathbb{X}_{xy})+1,$$
where $J$ is extended bilinearly over formal sums and differences.

\medskip

Similarly, the set of $Y$ markings can be thought of as a set $\mathbb{Y}_{xy} = \{Y_i\}_{i=1}^n$ of points in the plane with half-integer coordinates.   For an $\ell$-component link, define the $\ell$-tuple of Alexanders gradings by the formula

$$A^i_{xy}({\bf s})=J({\bf s}-\frac{1}{2}(\mathbb{Y}_{xy}+\mathbb{X}_{xy}),\mathbb{Y}^i_{xy}-\mathbb{X}^i_{xy})-\frac{n_i-1}{2},$$
where the $\mathbb{X}^i_{xy} \subset \mathbb{X}_{xy}$ and $\mathbb{Y}^i_{xy} \subset \mathbb{X}_{xy}$ are the subsets corresponding to the $i^{th}$ component of the link.  For links, the $A^i_{xy}$'s can take on half-integer values.

\medskip

The differential $\partial^-$ of this graded chain complex counts ``empty'' rectangles between two states.   For the sake of easily defining rectangles on $G_{xy}$, we will consider $G_{xy}$ for a moment as a torus by gluing the segment $\{(0,n)\}\times[0,n]$ to the segment $\{(0,0)\}\times [0,n]$, and similarly gluing the outermost and innermost segments of the grid parallel to the $x$-axis together as well.  Let ${\bf s}$ and ${\bf t}$ be states of a grid diagram $G_{xy}$.  A {\em rectangle $r$ connecting {\bf s} to {\bf t}} is if, after thinking of $G_{xy}$ as a torus, a rectangle on the torus that satisfies:

\begin{itemize}
\item ${\bf s}$ and ${\bf t}$ agree  along all but two grid lines parallel to the $x$-axis,\\

\item all four corners of $r$ are points are  in ${\bf s}\cup{\bf t}$,\\

\item by traversing an $x$-axis parallel boundary segment of $r$ in the direction indicated by the orientation inherited from the grid, then the segment is oriented from ${\bf s}$ to ${\bf t}$.
\end{itemize}

Note that if ${\bf s}, {\bf t}\in S_{xy}$ agree along all but two grid lines parallel to the $x$-axis, then there are exactly two rectangles satisfying the above conditions on the torus (cf. \cite{most}).  A rectangle $r$ is {\it empty} if Int$(r)\cap{\bf s}=\emptyset$. The set of all empty rectangles connecting ${\bf s}$ to ${\bf t}$ is denoted Rect$_{xy}^\circ({\bf s},{\bf t})$.  Note that Rect$_{xy}^\circ({\bf s},{\bf t})=\emptyset$ for all states ${\bf s}$ and ${\bf t}$ that do not agree on all but exactly two grid lines parallel to the $x$-axis.

\medskip
Let $R_{xy}$ be the polynomial algebra over $\mathbb{Z}/2\mathbb{Z}$ generated by the set of elements $\{\mathtt{\bf X}_i\}_{i=1}^n$ that are in one-to-one correspondence with $\mathbb{X}_{xy}=\{X_i\}_{i=1}^n$.  This ring has a Maslov grading so that the constant terms are in Maslov grading zero and the $\mathtt{\bf X}_i$ are in grading $-2$.  The Alexander multi-filtration is defined  so that the constant terms are filtration level zero and the variables ${\bf X}_k$ corresponding to the $i^{th}$  component of the link drop the $i^{th}$  multi-filtration level by one and preserve all others.
\medskip

Define the differential by
$$\partial^-_{xy}({\bf s})=\sum_{{\bf t}\in S_{xy}}\sum_{r\in\mathrm{Rect}_{xy}^\circ({\bf s},{\bf t})}\mathtt{\bf X}_1^{X_1(r)}\cdots \mathtt{\bf X}_n^{X_n(r)}\cdot{\bf t},$$
where $X_i(r)$ is the count of how many times the marking $X_i$ appears in $r$.

\medskip

Lemma 2.11 in \cite{most} states that if two $X$ markings $X_i$ and $X_k$ correspond to the same component of a $\ell$-component link, then the multiplication by $\mathtt{\bf X}_i$ is filtered chain homotopic to multiplication by $\mathtt{\bf X}_k$.  Thus, the homology of the complex $C^-(G_{xy})$ is a module over $\mathbb{Z}/2\mathbb{Z}[\mathtt{\bf X}_1,\dots,\mathtt{\bf X}_\ell]$, where $\mathtt{\bf X}_1,\dots,\mathtt{\bf X}_\ell$ correspond to $l$ different $X$-markings, each in a different component in the link.

\medskip

This construction gives a well-defined chain complex and leads to the following theorem of Manolescu, Ozsv\'ath, and Sarkar \cite{mos} (see also \cite{most}).

\begin{theorem}
\label{MOS} Let $G_{xy}$ be a grid presentation of a link $L$. The data $(C^-(G_{xy}),\partial^-)$ is a chain complex for the Heegaard-Floer homology $CF^-(S^3)$, with grading induced by $M_{xy}$, and the filtration level induced by $A_{xy}$ coincides with the link filtration of $CF^-(S^3)$.
\end{theorem}

\medskip

Suppose that the oriented link $L$ has has $\ell$ components.  Choose an ordering on $\mathbb{X}_{xy}$ so that for $i=1,\dots,\ell$, $X_i$ corresponds to the $i^{th}$ component of $L$.  Then setting all variables $\mathtt{\bf X}_i=0$ results in the hat version of knot Floer homology:
$$H_*(C^-(G_{xy})/\{\mathtt{\bf X}_i=0\}_{i=1}^n)\cong\widehat{HFK}(L)\otimes \bigotimes_{i=1}^\ell V_i^{\ot(n_i-1)},$$
where $V_i$ is the two-dimensional bigraded vector space spanned by one generator in bigrading $(0,0)$ and one generator in Maslov grading $-1$ and Alexander multi-grading corresponding to minus the $i^{th}$ basis vector.

\medskip
This theorem was proven by showing that homology was invariant under grid diagram moves (commutation and stabilization moves).   We will use this fact to prove that the hypercube homology is also invariant under hypercube diagram moves.

\medskip

\bigskip
\subsection{Hypercube homology from hypercube diagrams}  In this section we build a chain complex $C^-(H\Gamma)$ on a hypercube $H\Gamma$, define a differential $\partial^-$ on the chain complex, and prove that the hypercube homology $HFK^-(H\Gamma)$ is invariant with respect to hypercube moves.  Throughout this section, let $H\Gamma$ be a hypercube diagram of size $n$.  Let $G_{wx}$ and $G_{yz}$ be the oriented grid diagrams associated with $H\Gamma$.

\medskip

Let $S_{wx}$ and $S_{yz}$ be the set of states for the chain complexes $(C(G_{wx})^-,\partial^-_{wx})$ and $(C(G_{yz})^-,\partial^-_{yz})$ respectively (following the notation from the previous section).  These set of states are only defined on the 2-dimensional oriented grid diagrams---so  a state in $S_{xw}$ is a set of $n$ ordered pairs of a $xw$-coordinate system and a state in $S_{yz}$ is a set of $n$ ordered pairs of a $yz$-coordinate system.

\medskip

Denote the set of generators of $C^-(H\Gamma)$ by $S$.  We use $S_{wx}$ and $S_{yz}$ to build a set of states $S$ in the hypercube.   Like its grid diagram counterparts,  a set of states can be defined in $H\Gamma$ where each element state is a set of $n$ integer lattice points in $H\Gamma$ with no coordinate equal to $n$ satisfying the condition that no two points in the state determine a line parallel to one of the four axes.  There are $(n!)^3$ such configurations in $H\Gamma$, call this set of configurations $P$.  We need a set of states that only has $(n!)^2$ states, i.e., $|S|=(n!)^2$.  The solution is restrict ${\bf s}\in P$ so that for each point in ${\bf s}$, the $x$ and $z$-coordinates are equal:
$$S=\left\{ {\bf s} \in P \, | \, \forall (w_i,x_i,y_i,z_i) \in {\bf s}, x_i=z_i\right\}.$$
Define a map $\psi:S_{wx}\times S_{yz}  \ra S$ by the following procedure.  For ${\bf s}\in S_{wx}$, order all of the points in ${\bf s}$ by the last coordinate, i.e., write ${\bf s} = \{ (w_0,0),(w_1,1),\dots, (w_{n-1},n-1)\}$.  Similarly, order the points of ${\bf t}\in S_{yz}$ so that ${\bf t} = \{ (y_0,0),(y_1,1), \ldots, (y_{n-1},n-1)\}$.
With these orderings, define
$$\psi({\bf s},{\bf t}) = \left\{ (w_0,0,y_0,0), (w_1,1,y_1,1),\ldots (w_{n-1},n-1,y_{n-1},n-1)\right\}.$$

\begin{lemma} The map $\psi:S_{wx}\times S_{yz} \ra S$ is a bijection on sets.  In particular, $|S|=(n!)^2$. \label{lemma_bijectionmap}
\end{lemma}

\begin{proof} We use the proof mostly to describe a pair of functions on sets.  Define $\pi_{wx}:S\ra S_{wx}$ by
$$\pi_{wx}(\{(w_0,x_0,y_0,z_0),\dots,(w_{n-1},x_{n-1},y_{n-1},z_{n-1})\})=\{(w_0,x_0),\dots,(w_{n-1},x_{n-1})\},$$
and similarly define $\pi_{yz}:S\ra S_{yz}$ by projecting each point in a  state to the $yz$-plane.  Then it is clear that $\psi\circ (\pi_{wx}\times \pi_{yz}) = Id$ and $(\pi_{wx}\times \pi_{yz})\circ \psi=Id.$

\end{proof}

The complex $C^-(H\Gamma)$  has a grading and filtration determined by two functions on $S$.  Define the {\em Maslov grading} on $C^-(H\Gamma)$ by
$$M({\bf s}) = M_{wx}(\pi_{wx}({\bf s})) + M_{yz}(\pi_{yz}({\bf s})),$$
and define the {\it Alexander grading} by
$$A^i({\bf s}) = A^i_{wx}(\pi_{wx}({\bf s})) + A^i_{yz}(\pi_{yz}({\bf s})),$$
where the $i^{th}$ components are both projections from the same component of the hyperlink in the hypercube $H\Gamma$.

\medskip
To define the differential, we need to describe hyperrectangles in the hypercube $H\Gamma$ of size $n$. For the sake of easily defining hyperrectangles on $H\Gamma$, we will consider $H\Gamma$ for a moment as a $4$-torus by taking the natural quotient.  In $H\Gamma$, thought of as the fundamental domain of the $4$-torus,  $[a,b]\times [c,d] \times [0,n] \times [0,n]$ where $a,b,c,d \in [0,n)$ are integers, is an example of a possible hyperrectangle.  But in the $4$-torus, there is a ``complementary'' hyperrectangle $[b, a+n]\times [d,c+n] \times [0,n] \times [0,n]$ as well.  That complementary rectangle, thought of in the fundamental domain $H\Gamma$, will also be considered a hyperrectangle in $H\Gamma$.

\medskip

 Let  ${\bf s},{\bf t} \in S$ be two hypercube states.  A {\em $wx$-hyperrectangle $r$ connecting ${\bf s}$ to ${\bf t}$} is, after thinking of $H\Gamma$ as a $4$-torus, a hyperrectangle  such that

\begin{itemize}
\item ${\bf s}$ and ${\bf t}$ agree  along all but two grid lines parallel to the $x$-axis,\\

\item $\pi_{yz}({\bf s})=\pi_{yz}({\bf t})$,\\

\item in the projection of the hyperrectangle  to a rectangle $\pi_{wx}(r)$ in $G_{wx}$, all four corners of the rectangle  are points are  in $\pi_{wx}({\bf s})\cup \pi_{wx}({\bf t})$,\\

\item by traversing an $x$-axis parallel boundary segment of $\pi_{wx}(r)$ in the direction indicated by the orientation inherited from the grid, then the segment is oriented from $\pi_{wx}({\bf s})$ to $\pi_{wx}({\bf t})$.
\end{itemize}

Note that if ${\bf s}, {\bf t}\in S$ agree along all but two grid lines parallel to the $x$-axis, then there are exactly two rectangles satisfying the above conditions on the $4$-torus.  A rectangle $r$ is {\it empty} if Int$(r)\cap{\bf s}=\emptyset$. The set of all empty rectangles connecting ${\bf s}$ to ${\bf t}$ is denoted HRect$_{wx}^\circ({\bf s},{\bf t})$.  Define HRect$_{yz}^\circ({\bf s},{\bf t})$ similarly by interchanging the letters $w,x$ with $y,z$ respectively in the definition above.  Note that when HRect$_{yz}^\circ({\bf s},{\bf t})$ is non-empty, then HRect$_{wx}^\circ({\bf s},{\bf t})$ is, and vice-versa.

\medskip

Define two sets:   the set of variables $\{\mathtt{\bf W}_i\}_{i=1}^n$ that are in one-to-one correspondence with $\mathbb{W}=\{W_i\}_{i=1}^n$, the $W$-markings in $H\Gamma$ and the set of variables  $\{\mathtt{\bf Y}_i\}_{i=1}^n$ that are in one-to-one correspondence with $\mathbb{Y}=\{Y_i\}_{i=1}^n$, the $Y$-markings.   Let $R$ be the polynomial algebra over $\mathbb{Z}/2\mathbb{Z}$ generated by the set of elements $\{\mathtt{\bf W}_i\}_{i=1}^n$ and $\{\mathtt{\bf Y}_i\}_{i=1}^n$.  This ring has a Maslov grading so that the constant terms are in Maslov grading zero and the $\mathtt{\bf W}_i$ and $\mathtt{\bf Y}_i$  are in grading $-2$.  The Alexander filtration is defined  so that the constant terms are filtration level zero and the variables drop the filtration by one.
\medskip

For a $wx$-rectangle $r$ in $H\Gamma$, let $W_i(r)$ count the number of times the marking $W_i$ appears inside $r$.  Similarly define $Y_i(r)$.    The differential of the chain complex $\partial^-:C(H\Gamma) \ra C(H\Gamma)$ is given by
$$\partial^-({\bf s})=\sum_{{\bf t}\in S} \left(\sum_{r\in\mathrm{HRect}_{wx}^\circ({\bf s},{\bf t})}\mathtt{\bf W}_1^{W_1(r)}\cdots \mathtt{\bf W}_n^{W_n(r)}\cdot{\bf t} + \sum_{r'\in\mathrm{HRect}_{yz}^\circ({\bf s},{\bf t})}\mathtt{\bf Y}_1^{Y_1(r')}\cdots \mathtt{\bf Y}_n^{Y_n(r')}\cdot{\bf t}\right).$$

\medskip

It is clear that $\partial^-$ drops the Maslov index by 1 while preserving the Alexander multi-filtration.  Define $CH^-(H\Gamma)=H_* (C^-(H\Gamma),\partial^-)$ to be the cube homology of the hypercube diagram $H\Gamma$.

\subsection{Hypercube homology as knot Floer homology}  The definitions of and the notations for oriented grid diagrams, cube diagrams, and hypercube diagrams preceding this section were carefully described so as to make the proof of the invariance of $CH^-(H\Gamma)$  obvious from the isomorphism of $CH^-(H\Gamma)$ to the tensor product of the knot Floer homologies of $HFK^-(L_1)$ and $HFK^-(L_2)$, where $L_1$ and $L_2$ are the links associated to the oriented grid diagrams $G_{wx}$ and $G_{yz}$ of $H\Gamma$.

\bigskip

\begin{theorem}
\label{compthm}
Let $H\Gamma$ be a hypercube diagram and $G_{wx}$ and $G_{yz}$ be the oriented grid diagrams associated to $H\Gamma$. Let $(C^-(G_{wx}),\partial_{wx}^-)$ and $(C^-(G_{yz}),\partial_{yz}^-)$ be the  chain complexes associated to $G_{wx}$ and $G_{yz}$ respectively. Then
$$(C^-(H\Gamma),\partial^-)\cong(C^-(G_{wx}),\partial_{wx}^-)\otimes(C^-(G_{yz}),\partial_{yz}^-).$$
\end{theorem}

\medskip

\begin{proof} Let ${\bf s}\in S_{wx}$ be a grid state for $G_{wx}$ and ${\bf t}\in S_{yz}$ be a grid state for $G_{yz}$.  Recall the maps from the proof of Lemma~\ref{lemma_bijectionmap}: $\psi, \pi_{wx}$ and $\pi_{yz}$.  Use $\psi$ to define a map
$$\Psi:C^-(G_{wx})\ot C^-(G_{yz}) \ra C^-(H\Gamma)$$
given by ${\bf s} \ot {\bf t} \mapsto \psi({\bf s}, {\bf t})$.  Since $\Psi$ is a bijection on the generating sets by Lemma~\ref{lemma_bijectionmap}, it extends to an isomorphism. The map $\Psi$ clearly preserves the Maslov and Alexander gradings since the Maslov and Alexander gradings of $C^-(H\Gamma)$ were defined as the sum of the Maslov and Alexander gradings of $G_{wx}$ and $G_{yz}$.

\medskip

Furthermore, the map extends so that  $\mathtt{\bf W}_i$ variables in $C^-(G_{wx})$ are mapped to the $\mathtt{\bf W}_i$ variables in $C^-(H\Gamma)$ (since there is a unique $W$ marking in $H\Gamma$ for every $W$ marking in $G_{wx}$.  Similarly, the $\mathtt{\bf Y}_i$ variables in $C^-(G_{yz})$ are mapped to the  $\mathtt{\bf Y}_i$ in $C^-(H\Gamma)$.  Thus, for a $wx$-hyperrectangle $r$, counting  $W_i(r)$ inside of $r$ is equal to counting $W_i(\pi_{wx}(r))$ for the rectangle $\pi_{wx}(r)$ in $G_{wx}$.  A similar statement holds for counting $yz$-hyperrectangles.     Observe that
$$\partial_{wx}^-\otimes\partial_{yz}^-({\bf s}\otimes {\bf t})  =  \partial_{wx}^-({\bf s})\otimes {\bf t} + {\bf s}\otimes\partial_{yz}^-({\bf t}).$$
The summand $\partial_{wx}^-({\bf s})\ot {\bf t}$ counts only empty rectangles in $G_{wx}$.  For an empty rectangle connecting ${\bf s}$ to some other state ${\bf s'}$ in $G_{wx}$, $\pi_{yz}(\Psi({\bf s})\ot {\bf t}))=\pi_{yz}(\Psi({\bf s'} \ot {\bf t}))$, which means that  $$\mbox{HRect}_{yz}^\circ(\Psi({\bf s}\ot {\bf t}),\Psi({\bf s'} \ot {\bf t}))$$ is empty.  A similar statement holds for empty rectangles in $G_{yz}$.  Therefore  empty rectangles in $G_{wx}$ correspond to empty $wx$-hyperrectangles in $H\Gamma$ where the count of $W$-markings is the same for each.  There is a similar correspondence between empty rectangles and the count of $Y$ markings in $G_{yz}$.  Hence,
\begin{eqnarray}
\nonumber{\Psi(\partial_{wx}^-\otimes\partial_{yz}^-({\bf s}\otimes {\bf t}))} &=& \Psi(\partial_{wx}^-({\bf s}) \otimes {\bf t}) + \Psi({\bf s}\otimes\partial_{yz}^-({\bf t}))\\ \nonumber
&=&   \Psi\left(\sum_{{\bf s'}\in S_{wx}}\sum_{r\in\mathrm{Rect}_{wx}^\circ({\bf s},{\bf s'})}\mathtt{\bf W}_1^{W_1(r)}\cdots \mathtt{\bf W}_n^{W_n(r)}\cdot {\bf s'}\ot {\bf t}\right) + \\ \nonumber
&\mbox{}& \hspace{1cm} \Psi\left(\sum_{{\bf t'}\in S_{yz}}\sum_{r\in\mathrm{Rect}_{yz}^\circ({\bf t},{\bf t'})}\mathtt{\bf Y}_1^{Y_1(r)}\cdots \mathtt{\bf Y}_n^{Y_n(r)}\cdot {\bf s}\ot {\bf t'}\right) \nonumber\\
\nonumber &=&\sum_{{\bf u}\in S} \left(\sum_{r\in\mathrm{HRect}_{wx}^\circ(\Psi({\bf s} \ot {\bf t}),{\bf u})}\mathtt{\bf W}_1^{W_1(r)}\cdots \mathtt{\bf W}_n^{W_n(r)}\cdot {\bf u} \right) + \\
&\mbox{} & \nonumber \hspace{1cm} \sum_{{\bf u'}\in S}\left(\sum_{r'\in\mathrm{HRect}_{yz}^\circ(\Psi({\bf s} \ot {\bf t}),{\bf u'})}\mathtt{\bf Y}_1^{Y_1(r')}\cdots \mathtt{\bf Y}_n^{Y_n(r')}\cdot{\bf u'}\right)\nonumber\\
&=& \partial^-(\Psi({\bf s} \ot {\bf t})).\nonumber
\end{eqnarray}
\medskip

Thus $C^-(H\Gamma)\cong C^-(G_{wx})\otimes C^-(G_{yz})$ as $R$-modules with $\Psi \circ (\partial^-_{wx} \ot \partial^-_{yz}) = \partial^- \circ \Psi$.
\end{proof}

\medskip

Because of Theorem \ref{MOS}, we can conclude that the filtered chain homotopy type of $(C^-(H\Gamma),\partial^-)$ is invariant under hypercube commutation moves and hypercube stabilization moves.  A little more work establishes:

\medskip

\begin{theorem} \label{invariant}
Let $H\Gamma$ be a $4$-dimensional hypercube diagram. Then $CH^-(H\Gamma)$ is a hypercube invariant, that is, it is invariant under all three hypercube moves.  In particular, $$CH^-(H\Gamma) \cong HFK^-(L_1)\ot HFK^-(L_2),$$
where $L_1$ is the link represented by the oriented grid diagram $G_{wx}$ and $L_2$ is the link represented by the oriented grid diagram $G_{yz}$.
\end{theorem}

\medskip

\begin{proof}
The cube homology $CH^-(H\Gamma)$ is already invariant under cube stabilizations and cube commutations by Theorem~\ref{compthm}.  We need only show that $CH^-(H\Gamma)$ is invariant under a hypercube swap move.  A hypercube swap move is an orientation preserving map $SW:H\Gamma \ra H\Gamma$ given by $SW(w,x,y,z) = (y,z,w,x)$ on the underlying cube in $\BR^4$.  It also maps markings by $W\mapsto Y$, $X\mapsto Z$, $Y\mapsto W$, and $Z\mapsto X$.   As was explained in the hypercube move section, a hypercube swap exchanges $G_{wx}$ and $G_{yz}$ and also exchanges $G_{xy}$ and $G_{zw}$.

\medskip

We need to show that it swaps state systems as well.  First, $SW$ exchanges the $x$-coordinate with $z$-coordinate for each point in ${\bf s}\in S$, so it extends to a well-defined map ${SW:S\ra S}$ on the set of states.  Set $SW':S_{wx}\ot S_{yz} \ra S_{yz} \ot S_{wx}$ by setting $SW'({\bf s} \ot {\bf t}) = {\bf t}\ot {\bf s}$.  Both maps extend to maps on $C^-(H\Gamma)$ and $C^-(G_{wx})\ot C^-(G_{yz})$ respectively.  Furthermore, it is clear that $SW$ commutes with $\Psi$ in the sense that on generators, $$\Psi(SW'({\bf s} \ot {\bf t})) = SW(\Psi({\bf s} \ot {\bf t}).$$
The map $SW'_*:C^-(G_{wx})\ot C^-(G_{yz}) \ra C^-(G_{yz}) \ot C^-(G_{wz})$ clearly commutes with the differentials.  Putting the three maps together gives the desired isomorphism $SW_*:CH^-(H\Gamma)\ra CH^-(H\Gamma)$.  The case of a component swap is similar.
\end{proof}

\medskip

From the proofs of Theorem~\ref{compthm} and Theorem~\ref{invariant} we see that the filtered chain homotopy type of $C^-(H\Gamma,\partial^-)$ is also a hypercube invariant.

\medskip

By setting each of the $\mathtt{\bf W}_i$ and $\mathtt{\bf Y}_i$ variables to $0$, we get the ``hat'' version of knot Floer homology:
$$\widetilde{CH}(H\Gamma,n)=H_*(C^-(H\Gamma)/\{\mathtt{\bf W}_i=\mathtt{\bf Y}_i=0\}_{i=1}^n).$$

\begin{corollary}
Let $H\Gamma$ be a $4$-dimensional hypercube diagram of size $n$ with $\ell$ oriented components. Choose an ordering on $\{\mathtt{\bf W}_i\}_{i=1}^n$ so that for $i=1,\dots, \ell$, $W_i$ corresponds to the $i^{th}$ component of $H\Gamma$.  Similarly, order $\{\mathtt{\bf Y}_i\}_{i=1}^n$.    Let $L_1$ be the link represented by $G_{wx}$ and $L_2$ the link represented by $G_{yz}$.  Then
$$\widetilde{CH}(H\Gamma,n)\cong \widehat{HFK}(L_1)\otimes \widehat{HFK}(L_2) \otimes \bigotimes_{i=1}^\ell  V_i^{\otimes 2n_i-2},$$
where $V_i$ is a $2$-dimensional vector space spanned by one generator in zero Maslov and Alexander multi-gradings and the other in Maslov grading minus one and Alexander multi-grading corresponding to minus the $i^{th}$ vector.
\end{corollary}

Define $\widehat{CH}(H\Gamma) = \widehat{HFK}(L_1)\otimes \widehat{HFK}(L_2)$.

\bigskip

\subsection{Alexander polynomial of hypercube homology}
Let $H\Gamma$ be a hypercube diagram with $\ell$ components.  Let $t=(t_1,\dots, t_\ell)$ be a collection of variables, and for $\vec{s}=(s_1,\dots,s_\ell)\in (\frac12 \BZ)^\ell$, define $t^{\vec{s}}=t^{s_1}_1\dots t^{s_\ell}_\ell$.  For multi-graded groups $C_i(\vec{s})$ with Maslov grading $i$ and Alexander grading $\vec{s}$, define $$\chi(C;t)=\sum_{i,\vec{s}}(-1)^it^{\vec{s}}\mbox{rank}(C_i(\vec{s})).$$

\medskip

Following \cite{most}, we see that:

\medskip

\begin{theorem}
For any hypercube $H\Gamma$, let $L_1$ and $L_2$ be the oriented links represented by oriented grid diagrams $G_{wx}$ and $G_{yz}$.  The Euler characteristic of $\widehat{CH}$ is, up to sign,
$$\chi(\widehat{CH}(H\Gamma)) = \left\{\begin{array}{ll}
\pm \prod_{i=1}^{\ell} (t_i-2+t_i^{-1})\Delta_A(L_1;t)\cdot \Delta_A(L_2;t)  & \ell>1\\
\pm \Delta_A(L_1;t)\cdot \Delta_A(L_2;t) & \ell=1\end{array}\right.$$
where $\Delta_A(L_1,t)$ and $\Delta_A(L_2,t)$ are multivariable Alexander polynomials, normalized so that they are symmetric up to sign under the involution of sending $t_i$ to their inverses. \label{hat_Euler}
\end{theorem}

\medskip
Note that the same variables are used in both Alexander polynomials (there are only $\ell$ variables, not $2\ell$).  This is because each component in $H\Gamma$ gives rise to a component in $L_1$ and $L_2$.  The identification of variables corresponds to this identification of components in $L_1$ and $L_2$.

\bigskip
\section{Hopf linked tori and other examples}
\label{exampleofhypercubes}
\bigskip

In this section we provide specific examples of embedded tori and  calculate their hypercube homology invariants.  The embedding problem for interesting knotted tori is nontrivial.  It is fairly easy to create embedded knotted tori with hypercube diagrams $H\Gamma(L_1,L_2)$ where $L_1$ is any knot and $L_2$ is the unknot.    For example, the hypercube diagram $H\Gamma(L_1, L_2)$ pictured in Figure~\ref{torusproof1} is an embedded torus (as you can easily check that there are no vertical double point circles).  The projection of $H\Gamma(L_1,L_2)$ to $G_{wx}$ is a trefoil and the projection of $H\Gamma(L_1,L_2)$ to $G_{yz}$ is the unknot (see Figure~\ref{Trefoilknot_ex} below).

\begin{figure}[H]
\includegraphics[scale=1]{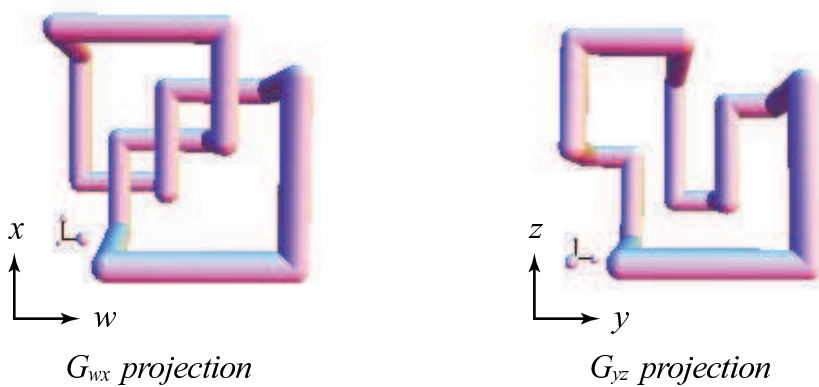}
\caption{The $G_{wx}$ and $G_{yz}$ grid projections of the hypercube $H\Gamma$ shown in Figure~\ref{torusproof1}.  Note that $\widehat{CH}(H\Gamma) = \widehat{HFK}(\mbox{trefoil knot})$.} \label{Trefoilknot_ex}
\end{figure}
\medskip

Interesting embedded knotted tori, i.e., those where $L_1$ or $L_2$ are not the unknot, are more difficult to find.  In this paper we describe a simple example, which turns out to be a link of two tori.  Like the example described above, it is easy to find a (small sized) hypercube diagram of embedded tori where the projection to $G_{wx}$ is the Hopf link and the projection to $G_{yz}$ is a split link of two unknots.  Starting with a standard torus embedding, the three examples below build up to an example of a hypercube that represents two embedded linked tori such that the hypercube projects to the Hopf link in {\em both} projections.   We will call this last example {\em Hopf linked tori}.  This example, which shows that it is possible to construct a hypercube diagram in which both projections are not unknots or split links of unknots, indicates that hypercube diagrams represent a large and interesting subset of embedded tori in $\BR^4$.

\medskip

We also calculate the Euler characteristic of the cube homology for the three examples.  It is interesting that the cube homology invariants distinguish the second and third examples, both of which can be thought of as standard tori ``linked'' in different ways.  It may be interesting to study the different linking numbers for surfaces in $\BR^4$ in terms of hypercube diagrams and hypercube homology (cf. \cite{FR}, see also \cite{Kirk}, \cite{Li}).

\medskip

Finally, as an indication of the interesting genera that can occur with hypercube diagrams, we present an example of an immersed torus represented by a hypercube diagram where the projections to $G_{wx}$ is a trefoil and the projection to $G_{yz}$ is  the $5_2$ knot.  Similar to the Hopf link torus above, it is quite likely that a computer search of hypercube diagrams in which one projection is a size  14 stabilized trefoil grid diagram and the other is a size 14 stabilized $5_2$ knot will yield a hypercube that represents an embedded torus in $\BR^4$.

\bigskip

\noindent {\bf Example 1.  The standard embedded torus}.  The standard torus is one of the few examples that can be easily visualized (see Figure~\ref{std_torus}).

\medskip

\begin{figure}[H]
\includegraphics[scale=1]{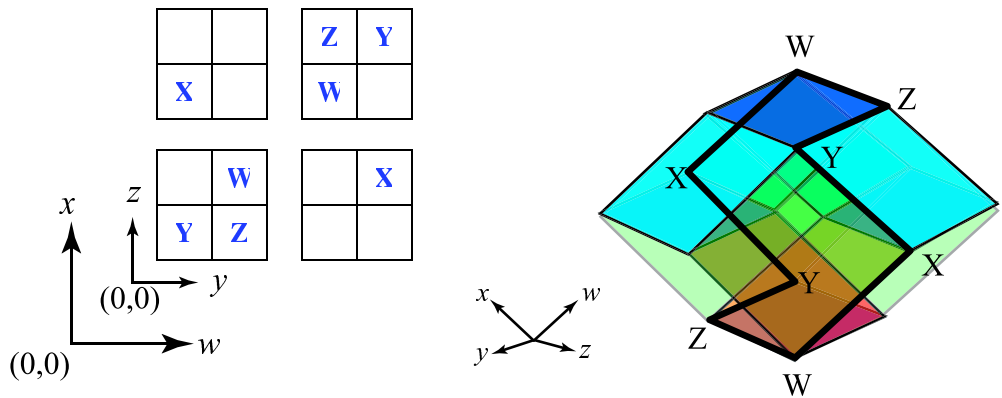}
\caption{A hypercube diagram of a standard torus on the left and its picture on the right. } \label{std_torus}
\end{figure}
\medskip

As in Figure~\ref{step4}, it is easier to see the torus above if part of it is removed:

\begin{figure}[H]
\includegraphics[scale=1]{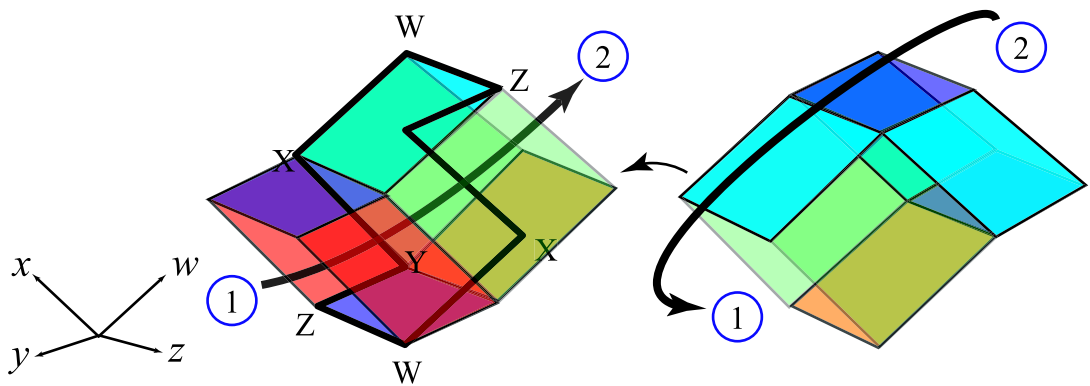}
\caption{A standard torus in $\BR^4$.  The loop going from 1 to 2 on the left and then from 2 to 1 on the right is a loop in $\BR^4$ that wraps around one of the $S^1$ factors of the torus. } \label{std_torus2}
\end{figure}
\medskip

The hypercube homology for the standard torus is  of the unknot:  $$\widehat{CH}(\mbox{standard torus})=\widehat{HFK}(\mbox{unknot})\ot \widehat{HFK}(\mbox{unknot}) \cong \BZ.$$

\bigskip

\noindent {\bf Example 2.  Once-linked standard tori.}  Below is a hypercube diagram schematic for two embedded tori such that the $G_{wx}$ projection is the Hopf link but the $G_{yz}$ projection is a split link of two unknots (see the projections on the right of Figure~\ref{oncelinkedtori}).

\begin{figure}[H]
\includegraphics[scale=1]{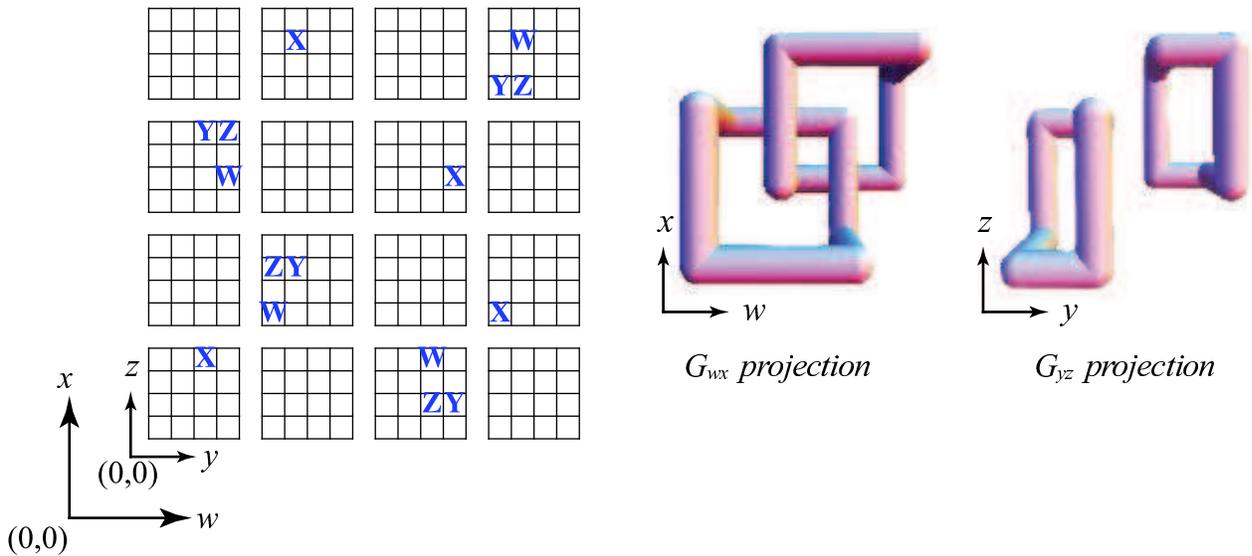}\caption{Once-linked tori.} \label{oncelinkedtori}
\end{figure}
\medskip

The hypercube homology for once-linked tori is the tensor product of $\widehat{HFK}$ of the Hopf link and $\widehat{HFK}$ of the split link of two unknots.  The Euler characteristic of the hypercube homology is zero.

\bigskip

\noindent{\bf Example 3. Hopf linked tori. } \  A computer program was used to search different hypercube diagrams with Hopf links in both the $G_{wx}$ and $G_{yz}$ projections.  The first example of such a hypercube diagram has size 8 (see the example in Figure~\ref{hopflinkedtori} below). While there may be smaller sized examples, of the millions of hypercube diagrams checked of size 7 or less for pairs of standard embedded tori, all were once-linked tori like the example above.  There are, however, plenty of {\em immersed} Hopf linked tori with hypercube diagrams of size 7 or less.    The difficulty is finding examples of hypercube diagrams that represent embedded tori (only horizontal or vertical double point circles but not both).

\medskip

\begin{figure}[H]
\includegraphics[scale=1]{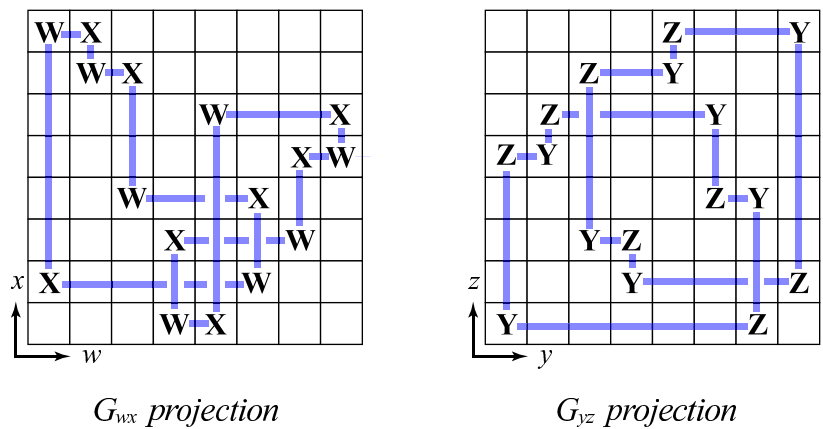}
\caption{The $G_{wx}$ and $G_{yz}$ projections of the hypercube diagram for Hopf Linked Tori.  Note the Hopf link in both projections.} \label{hopflinkedtori}
\end{figure}
\medskip

Finding an example of embedded Hopf linked tori is roughly equivalent to being able to find embedded knotted tori with two different knot types in the $G_{wx}$ and $G_{yz}$ projections.  Here is why:  Extensive computer experimentation with hypercube diagrams of embedded knotted tori shows that once the knot type of the $G_{wx}$ projection is fixed and the size of the hypercube is equal to the arc index of that knot type, then $G_{yz}$ projection is the unknot {\em with no crossings}.  By increasing the size of hypercube by stabilizing a few times, crossings in the $G_{yz}$ projection start to appear, but the crossing come in pairs of overcrossings or pairs of undercrossings (Type II Reidemeister moves on the unknot).  The general requirement needed to build any nontrivial knot in the $G_{yz}$ projection is the ability to get an overcrossing followed by an undercrossing.  This configuration is exactly what the Hopf linked tori above shows can be done in both projections.  Note that to get embedded Hopf linked tori, the stabilizations in the $G_{wx}$ projection were carefully chosen (study the $G_{wx}$ projection in Figure~\ref{hopflinkedtori}).   Stabilizing around crossings  in a similar way for grid diagrams of other knots should lead to interesting nontrivial embedded knotted tori.

\medskip

Figure~\ref{hopflinkedtori2} below is the hypercube diagram schematic for a Hopf linked tori.  It can be used to check that the hypercube diagram indeed represents two embedded tori (by checking for double point circles).

\medskip

\begin{figure}[H]
\includegraphics[scale=1]{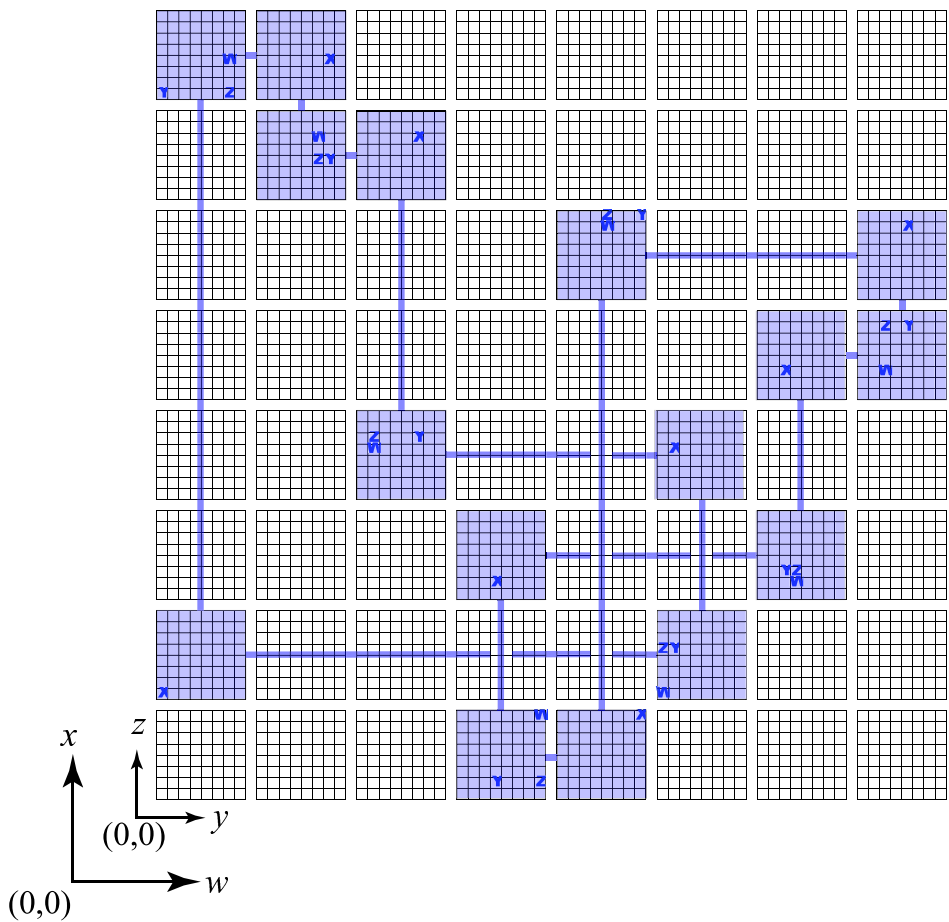}
\caption{The hypercube diagram schematic for an embedded Hopf linked tori (with the $G_{wx}$ projection overlaid upon the schematic).} \label{hopflinkedtori2}
\end{figure}
\medskip

The table below lists the $W$, $X$, $Y$, and $Z$ markings of the Hopf linked tori shown in Figure~\ref{hopflinkedtori2}.

\bigskip

\begin{center}
\begin{tabular}{|c|l|}

\hline
Marking & Points \\[.2cm] \hline
$W$ & $(\frac72, \frac12, \frac{15}{2}, \frac{15}{2}), (\frac{11}{2}, \frac32, \frac12, \frac12), (\frac{13}{2}, \frac52, \frac72, \frac32), (\frac52, \frac72, \frac32, \frac92),$ \\[.2cm]
& $(\frac{15}{2}, \frac92, \frac52,
   \frac52), (\frac92, \frac{11}{2}, \frac92, \frac{13}{2}), (\frac32, \frac{13}{2}, \frac{11}{2}, \frac{11}{2}), (\frac12, \frac{15}{2}, \frac{13}{2}, \frac72)$\\[.2cm] \hline
$X$ & $(\frac72, \frac52, \frac72, \frac32), (\frac{11}{2}, \frac72, \frac32, \frac92), (\frac{15}{2}, \frac92, \frac52, \frac52), (\frac52, \frac{15}{2}, \frac{11}{2}, \frac{11}{2}), $\\[.2cm]
& $(\frac{15}{2}, \frac{11}{2}, \frac92,
 \frac{13}{2}), (\frac92, \frac12, \frac{15}{2}, \frac{15}{2}), (\frac32, \frac{15}{2}, \frac{13}{2}, \frac72), (\frac12, \frac32, \frac12, \frac12)$\\[.2cm] \hline
$Y$ &  $(\frac72, \frac12, \frac72, \frac32), (\frac{11}{2}, \frac32, \frac32, \frac92), (\frac{13}{2}, \frac52, \frac52, \frac52), (\frac52, \frac72, \frac{11}{2}, \frac{11}{2})$ \\[.2cm]
& $(\frac{15}{2}, \frac92, \frac92, \frac{13}{2}), (\frac92, \frac{11}{2}, \frac{15}{2}, \frac{15}{2}), (\frac32, \frac{13}{2}, \frac{13}{2}, \frac72), (\frac12, \frac{15}{2}, \frac12, \frac12)$\\[.2cm] \hline
$Z$ & $(\frac72, \frac12, \frac{15}{2}, \frac32), (\frac{11}{2}, \frac32, \frac12, \frac92), (\frac{13}{2}, \frac52, \frac72, \frac52), (\frac52, \frac72, \frac32, \frac{11}{2}),$ \\[.2cm]
& $(\frac{15}{2}, \frac92, \frac72, \frac{13}{2}),  (\frac92, \frac{11}{2}, \frac92, \frac{15}{2}), (\frac32, \frac{13}{2}, \frac{11}{2}, \frac72), (\frac12, \frac{15}{2}, \frac{13}{2}, \frac12)$\\[.2cm] \hline
\end{tabular}
\end{center}

\bigskip

Finally, the ``hat'' version of hypercube homology for the Hopf linked tori is the tensor product of the ``hat'' version of knot Floer homology of two Hopf links.  The Euler characteristic is, by Theorem~\ref{hat_Euler},
$$\chi(\widehat{CH}(H\Gamma)) = \pm (t_1 -2 + t_1^{-1})(t_2 -2 +t_2^{-1}),$$
which shows that the embedded Hopf linked tori is different from the embedded once-linked tori (which has Euler characteristic zero).

\bigskip

\noindent{\bf Example 4.  An immersed torus knot that is an amalgamation of the Trefoil and the $5_2$ knot.}  We present the example in Figure~\ref{trefoil_K52_proj} to show how knotted tori can be constructed as an amalgamation of two different knots.

\begin{figure}[H]
\includegraphics[scale=1]{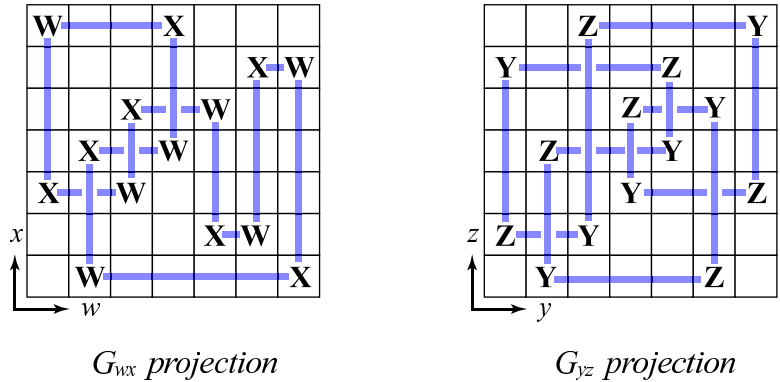}
\caption{The $G_{wx}$ and $G_{yz}$ projections for the immersed knotted torus that is an amalgamation of the Trefoil and the $5_2$ knot.} \label{trefoil_K52_proj}
\end{figure}
\medskip

It can be shown from the hypercube diagram schematic in Figure~\ref{trefoil_K52} that there are both horizontal and vertical double point circles, so this torus is immersed.

\begin{figure}[H]
\includegraphics[scale=1]{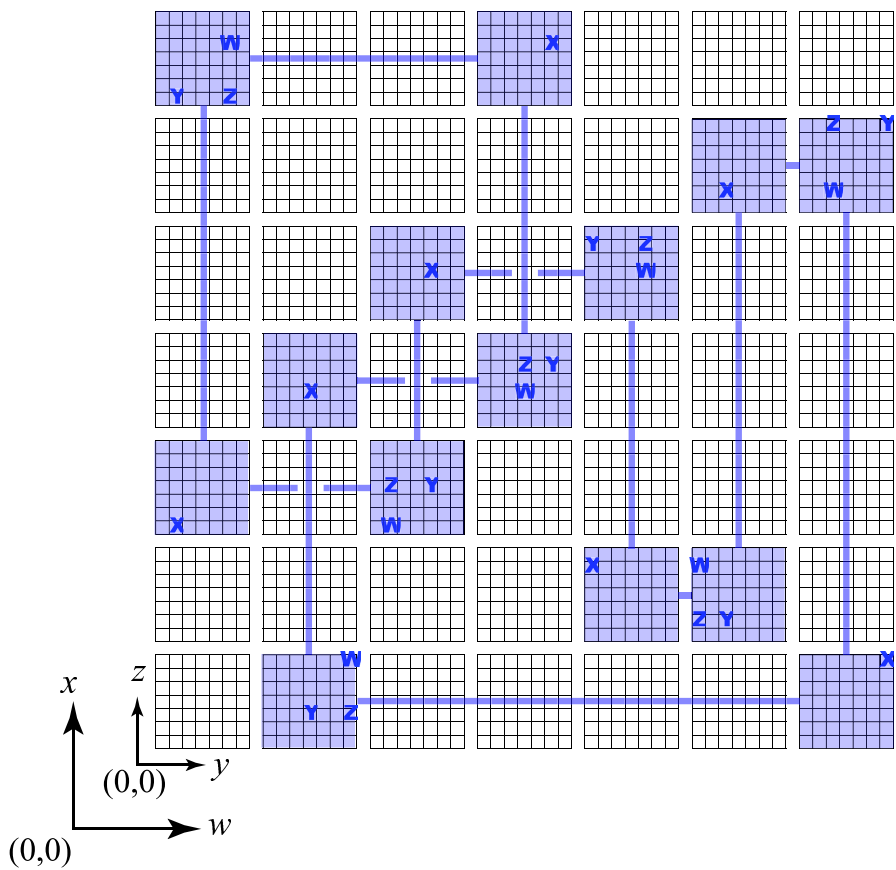}
\caption{The hypercube diagram schematic for the immersed knotted torus that is an amalgamation of the Trefoil and the $5_2$ knot.  The trefoil $G_{wx}$ projection is overlaid upon the schematic.} \label{trefoil_K52}
\end{figure}
\medskip

 By carefully stabilizing the trefoil knot grid diagram, it should be possible to find a different hypercube diagram that represents an embedded torus that is the amalgamation of the Trefoil and the $5_2$ knot.

\end{document}